\documentclass[11pt]{amsart}

\usepackage{enumitem}
\usepackage[
paper=letterpaper,
text={147mm,230mm},centering
]{geometry}
\usepackage{amssymb}
\usepackage{cancel}
\usepackage{csquotes}
\usepackage[dvipsnames]{xcolor}
\usepackage{marginnote}
\usepackage{color}
\usepackage{tabto}

\usepackage[color=yellow!90!black]{todonotes}
\usepackage{xparse}

\usepackage{placeins}

\usepackage{tikz}
\usetikzlibrary{decorations.markings}
\usetikzlibrary{positioning}
\usepackage[bookmarks]{hyperref}
\hypersetup{colorlinks=true,linkcolor=blue,citecolor=blue}

\iftrue
\makeatletter
\def\@settitle{%
  \vspace*{-20pt}
  \begin{flushleft}%
    \baselineskip14\p@\relax
    \normalfont\bfseries\LARGE
    \@title
  \end{flushleft}%
}
\def\@setauthors{%
  \begingroup
  \def\thanks{\protect\thanks@warning}%
  \trivlist
  \large \@topsep30\p@\relax
  \advance\@topsep by -\baselineskip
  \item\relax
  \author@andify\authors
  \def\\{\protect\linebreak}%
  \authors
  \ifx\@empty\contribs
  \else
    ,\penalty-3 \space \@setcontribs
    \@closetoccontribs
  \fi
  \normalfont
  \endtrivlist
  \endgroup
}
\def\@setaddresses{\par
  \nobreak \begingroup\raggedright
  \small
  \def\author##1{\nobreak\addvspace\smallskipamount}%
  \def\\{\unskip, \ignorespaces}%
  \interlinepenalty\@M
  \def\address##1##2{\begingroup
    \par\addvspace\bigskipamount\noindent
    \@ifnotempty{##1}{(\ignorespaces##1\unskip) }%
    {\ignorespaces##2}\par\endgroup}%
  \def\curraddr##1##2{\begingroup
    \@ifnotempty{##2}{\nobreak\noindent\curraddrname
      \@ifnotempty{##1}{, \ignorespaces##1\unskip}\/:\space
      ##2\par}\endgroup}%
  \def\email##1##2{\begingroup
    \@ifnotempty{##2}{\smallskip\nobreak\noindent E-mail address%
      \@ifnotempty{##1}{, \ignorespaces##1\unskip}\/:\space
      \ttfamily##2\par}\endgroup}%
  \def\urladdr##1##2{\begingroup
    \def~{\char`\~}%
    \@ifnotempty{##2}{\nobreak\noindent\urladdrname
      \@ifnotempty{##1}{, \ignorespaces##1\unskip}\/:\space
      \ttfamily##2\par}\endgroup}%
  \addresses
  \endgroup
  \global\let\addresses=\@empty
}
\def\@setabstracta{%
    \ifvoid\abstractbox
  \else
    \skip@25\p@ \advance\skip@-\lastskip
    \advance\skip@-\baselineskip \vskip\skip@
    \box\abstractbox
    \prevdepth\z@ 
    \vskip-10pt
  \fi
}

\let\oldtocsection=\tocsection
\let\oldtocsubsection=\tocsubsection
\let\oldtocsubsubsection=\tocsubsubsection
\renewcommand{\tocsection}[2]{\hspace{0em}\oldtocsection{#1}{#2}}
\renewcommand{\tocsubsection}[2]{\hspace{2em}\oldtocsubsection{#1}{#2}}
\renewcommand{\tocsubsubsection}[2]{\hspace{4em}\oldtocsubsubsection{#1}{#2}}

\renewenvironment{abstract}{%
  \ifx\maketitle\relax
    \ClassWarning{\@classname}{Abstract should precede
      \protect\maketitle\space in AMS document classes; reported}%
  \fi
  \global\setbox\abstractbox=\vtop \bgroup
    \normalfont\small
    \list{}{\labelwidth\z@
      \leftmargin0pc \rightmargin\leftmargin
      \listparindent\normalparindent \itemindent\z@
      \parsep\z@ \@plus\p@
      
    }%
    \item[\hskip\labelsep\bfseries\abstractname.]%
}{%
  \endlist\egroup
  \ifx\@setabstract\relax \@setabstracta \fi
}
\def\section{\@startsection{section}{1}%
  \z@{-1.2\linespacing\@plus-.5\linespacing}{.8\linespacing}%
  {\normalfont\bfseries\large}}
\def\subsection{\@startsection{subsection}{2}%
  \z@{-.8\linespacing\@plus-.3\linespacing}{.3\linespacing\@plus.2\linespacing}%
  {\normalfont\bfseries}}
\def\subsubsection{\@startsection{subsubsection}{3}%
  \z@{.7\linespacing\@plus.1\linespacing}{-1.5ex}%
  {\normalfont\itshape}}
\def\@secnumfont{\bfseries}
\makeatother
\fi 

\usepackage{tikz}
\usetikzlibrary{decorations.markings}
\usetikzlibrary{positioning}
\usepackage[bookmarks]{hyperref}
\hypersetup{colorlinks=true,linkcolor=blue,citecolor=blue}

\usepackage{mathptmx}
\usepackage{amssymb}
\usepackage[all]{xy}
\usepackage{graphicx,color,float}
\usepackage[bookmarks]{hyperref}
\usepackage{pinlabel}
\usepackage{amsmath}
\usepackage{mathtools}
\usepackage{tikz-cd}
\usepackage{adjustbox}
\usepackage{pb-diagram,pb-xy}
\textwidth=\paperwidth \advance\textwidth by-2\oddsidemargin
\advance\oddsidemargin by-1in
\evensidemargin=\oddsidemargin
\calclayout

\def\Z{\mathbb{Z}}

\def\R{\mathbb{R}}
\def\C{\mathbb{C}}

\def\d{\partial}

\def\tilde{\widetilde}

\def\sign{\operatorname{sign}}

\def\rank{\operatorname{rank}}

\def\CAT{\operatorname{CAT}}

\def\int{\operatorname{int}}
\def\+{\oplus}

\def\vol{\operatorname{vol}}

\def\inj{\textup{inj}}

\theoremstyle{plain}
\newtheorem{theorem}{Theorem}[section]
\newtheorem{proposition}[theorem]{Proposition}
\newtheorem{lemma}[theorem]{Lemma}

\theoremstyle{definition}
\newtheorem{definition}[theorem]{Definition}
\newtheorem{remark}[theorem]{Remark}

\theoremstyle{definition}

\theoremstyle{remark}

\newcommand{\nonamethmname}{}

\NewDocumentEnvironment{genthm}{O{plain}m}
 {\renewcommand{\nonamethmname}{#2}\begin{nonamethm#1}}
 {\end{nonamethm#1}}
\NewDocumentEnvironment{genthm*}{O{plain}mo}
 {\renewcommand{\nonamethmname}{#2}%
  \IfNoValueTF{#3}
    {\begin{nonamethm#1}\relax}%
    {\begin{nonamethm#1}[#3]}%
  \mbox{}}
 {\end{nonamethm#1}}
\def\to{\mathchoice{\longrightarrow}{\rightarrow}{\rightarrow}{\rightarrow}}
\makeatletter
\newcommand{\shortxra}[2][]{\ext@arrow 0359\rightarrowfill@{#1}{#2}}
\def\longrightarrowfill@{\arrowfill@\relbar\relbar\longrightarrow}
\newcommand{\longxra}[2][]{\ext@arrow 0359\longrightarrowfill@{#1}{#2}}

\makeatother

\makeatletter
\@namedef{subjclassname@2020}{%
  \textup{2020} Mathematics Subject Classification}
\makeatother

\begin{document}

\title
{Bounds on Cheeger-Gromov invariants and simplicial complexity of triangulated manifolds}


\author{Geunho Lim$\dag$}
\thanks{$\dag$ Partially supported by the National Research Foundation of Korea grant 2019R1A3B2067839.}
\address{Department of Mathematics, University of California, Santa Barbara, CA, United States}
\curraddr{Einstein Institute of Mathematics, Hebrew University of Jerusalem, Jerusalem, Israel}
\email{limg@ucsb.edu}

\author{Shmuel Weinberger$\dag\dag$}
\thanks{$\dag\dag$ Partially supported by the National Science Foundation grant DMS-2105451.}
\address{Department of Mathematics, University of Chicago, IL, United States}
\curraddr{}
\email{shmuel@math.uchicago.edu}

\thanks{}
\maketitle

\begin{abstract}
    We show the existence of linear bounds on Wall $\rho$-invariants of PL manifolds, employing a new combinatorial concept of $G$-colored polyhedra. As application, we show that how the number of h-cobordism classes of manifolds simple homotopy equivalent to a lens space with $V$ simplices and the fundamental group of $\mathbb{Z}_n$ grows in $V$. Furthermore we count the number of homotopy lens spaces with bounded geometry in $V$. Similarly, we give new linear bounds on Cheeger-Gromov $\rho$-invariants of PL manifolds endowed with a faithful representation also. A key idea is to construct a cobordism with a linear complexity whose boundary is $\pi_1$-injectively embedded, using relative hyperbolization. As application, we study the complexity theory of high-dimensional lens spaces. Lastly we show the density of $\rho$-invariants over manifolds homotopy equivalent to a given manifold for certain fundamental groups. This implies that the structure set is not finitely generated.
\end{abstract}

\tableofcontents

\section{Introduction and main results}

Invariants of the type we consider in this paper originally arose in the context of finite (quotient groups of) fundamental groups, in work of Browder-Livesay on classifying free involutions on the sphere~\cite{BL}. They were related by Hirzebruch~\cite{Hir68} to invariants of Atiyah–Bott~\cite{AB}, which they had used to show that two different lens spaces are not h-cobordant, and later by Atiyah-Patodi-Singer~\cite{APS} to $\eta$-invariants. For manifolds with a finite fundamental group, these invariants play a crucial role in augmenting characteristic classes to completely classify all manifolds within a given simple homotopy type as Wall observed explicitly in the case of homotopy lens spaces (see~\cite{Wall99},~\cite{CW21}).

Cheeger and Gromov~\cite{CG} introduced a natural $L^2$ analog of the above invariants for manifolds with an infinite fundamental group. These invariants, in addition to their original application in the integrality formula for specific characteristic classes of complete Riemannian manifolds with bounded curvature and finite volume, have been employed to show that certain structure sets of manifolds are infinite~\cite{CW}. We will see that similar arguments show that these structure sets are not only infinite but, also infinitely generated. Moreover, these invariants (for the Dirac operator replacing the signature operator) play a crucial role in proving that moduli spaces of metrics with positive scalar curvature have infinitely many components~\cite{PT}. Other applications include three-dimensional topological knot cobordism~\cite{COT} and the study of the complexity of 3-manifolds in various senses~\cite{Cha16}.

Some of these applications make use of an important inequality of Cheeger and Gromov~\cite{CG85} which holds for smooth Riemannian manifolds of bounded geometry:
\[
\rho^{(2)}(M) \leq C(d, \inj, K) \cdot \vol(M)
\]
where $d = \dim(M)$, $\inj$ is a lower bound on the injectivity radius, and $K$ is a bound on the absolute value of the sectional curvature of $M$. (We will refer to such bounds on a Riemannian metric as \emph{a bound on its geometry.})

Based on this inequality, Gromov conjectured~\cite{Gro82} that for manifolds with bounded geometry, if they are nullcobordant, then they bound other (bounded geometry) manifolds with at most a linear increase in volume. Such estimates would provide a topological explanation for the aforementioned inequality. The validity of this conjecture remains unknown, but some progress has been made in~\cite{CDMW17}.

In this paper, we will consider analogues of these invariants and prove analogous inequalities for PL manifolds, where the volume of the manifold is replaced by the total number of simplices it contains. Extending these invariants to PL (and even topological) manifolds is straightforward using a bordism argument presented in~\cite[Section 13]{Wall99}. Furthermore, these invariants can be generalized to Witt spaces (as demonstrated in~\cite{Siegel} and~\cite{ALMP}), and our results are expected to hold in this broader context as well. We intend to pursue this direction in a future paper. In contrast to the analytical approach employed by Cheeger and Gromov, we rely on purely geometric arguments.

This inequality is \emph{not} the natural analogue of the Cheeger-Gromov inequality for the PL setting. This is because we do \emph{not} impose any local condition on the geometry. In other words, the number of neighbors of a vertex can be unbounded in this collection of manifolds. Additionally, we don’t know what the smooth analogue of our inequality is. Perhaps it involves the “engulfing radius”, which is the largest $r$ such that metric balls of radius $r$ in $M$ are included in larger smooth contractible balls.

We will soon discuss the precise statement of our results and some applications. However, inspired by Gromov's heuristic reasoning, it is tempting to conjecture that the number of simplices in PL cobordism (and Witt cobordism) behaves linearly or perhaps almost linearly. Currently we have no idea how to prove such a result. A proof has been given by~\cite{MW} for PL manifolds of bounded geometry of the analogue of the result in~\cite{CDMW17} i.e. of almost linearity. The implied constants involved in this proof depend on the bounds imposed on the local geometry in an unknown manner.

We now state our main theorem (leaving the definition of the $\rho$-invariants to the next section).

\begin{theorem}\label{thm:AS}
If $M$ is an odd dimensional PL manifold and one is given a homomorphism $\alpha:\pi_1(M) \to G$ for a finite group $G$, then $\lvert \rho_{g}(M) \rvert \leq C(d) \cdot \lvert G \rvert  \cdot \Delta (M)$, for any nontrivial element $g$ in $G$.
\end{theorem}

where $C(d)$ is a constant just depending on the dimension, $\lvert G \rvert$ is the order of $G$, and $\Delta (M)$ is the number of simplices in a triangulation of $M$. See the discussion in section~\ref{sec:polyhedron}, for the optimality of various aspects of this inequality.

\begin{theorem}\label{thm:CG}
If $M$ is an $(4k-1)$-dimensional PL manifold endowed with a faithful representation of $\pi_1$, then the Cheeger-Gromov invariant satisfies $\lvert \rho^{(2)}(M) \rvert \leq C'(d) \cdot \Delta (M)$.
\end{theorem}

Since smooth manifolds with bounded geometry metrics have PL triangulations whose number of simplices is proportional to the volume, see~\cite{ChMS84}, the above inequality does imply the Cheeger-Gromov inequality. Note that the factor $\lvert G \rvert$, which is necessary in Theorem~\ref{thm:AS}, is not present in Theorem~\ref{thm:CG}. Either theorem can be applied to prove the following theorem.

\begin{theorem}\label{cor:CG}
The number of simplices of the standard lens space $L_{N}(1,1,\cdots,1)$ of dimension $2d-1$ (with fundamental group $\Z_N$) grows (in $N$) like $N^{d-1}$, i.e. the number of simplices is bounded above and below by dimensional constants times $N^{d-1}$.
\end{theorem}

We remark that the theorems remain true with their given proofs for arbitrary triangulations of manifolds which are not necessarily restricted to PL triangulations. We note that the proof for Theorem~\ref{cor:CG} when the dimension is $1$ mod $4$ is slightly indirect as in that case there is no Cheeger-Gromov invariant.

In dimension $3$, the results stated in Theorem~\ref{cor:CG} and Theorem~\ref{thm:CG} were originally established by Cha~\cite{Cha16}. The first author, in a subsequent work~\cite{Lim}, improved the estimate for $C_3'$, resulting in a more refined bound. This paper is the result of rethinking some of the algebra in those papers from a more geometrical perspective, leading to an improved estimate of $12$ for $C_3'$. (However, in high dimensions, the constants are much reasonable.)

In~\cite{CL}, an extension of Theorem~\ref{thm:CG} to other $L^2$ $\rho$-invariants will be given by a mixture of the algebraic and geometric ideas in all of these papers. The underlying philosophical insight is that approaching the geometric arguments from an algebraic perspective enables the establishment of functoriality, whereas a purely geometric approach would require the inclusion of a complexity measurement for the homomorphism. Theorem~\ref{thm:AS}, in the form given here, actually relies on methods and results of this work. The technique we use for Theorem~\ref{thm:CG} would give some constant for each finite $G$ but not linearity for $\lvert G \rvert$. We note that the applications below do not depend on this linearity.

The following is an application of our techniques to trying to understand the complexity of manifolds that are produced by surgery theory. The classification of manifolds homotopy equivalent to a lens space is one of the high points of Wall's book~\cite{Wall99} - we now estimate how many simplices are needed to construct those manifolds.

\begin{theorem}\label{cor:AS}
The number of h-cobordism classes of manifolds simple homotopy equivalent to a lens space $L_N^d$ with at most $V$ simplices is bounded above and below by constant multiples of $V^{[(N-1)/2] + \delta (N,d)}$, with constants depending on $d$, where $d$ is odd, and $\delta (N,d) = 0$ unless $N$ is even and $d$ is $3$ mod $4$, and in that case it is $1$.
\end{theorem}

This is based on the connection between $\rho$-invariants and surgery theory established in~\cite{Wall99}. Specifically, we balance our bound on the $\rho$-invariants of these manifolds against a construction given in~\cite{Wei82}. This technique allows us to construct that many homotopy lens spaces with $V$ simplices. The key implication of our findings is that the number of simplices required for an element of the homotopy structure set of a lens space~\cite{Wall99} is approximately equivalent to its distance from the origin within that set, when viewed as an abelian group.

Counting the manifolds seems more natural than counting $h$-cobordism classes of these manifolds. This would require estimates on Reidemeister torsions, which are analogous to the above inequalities for $\rho$-invariants. The idea of~\cite{Sou99} suggests that bounding the number of simplices suffices for this project. However it is not known how to do this in general. Using the result of~\cite{ChMS84} and these ideas, we have:

\begin{theorem}\label{cor:bddgeo}
The number of homotopy lens spaces with bounded geometry with fundamental group $\Z/N$ in dimension $d$ and volume $V$ can be bounded above and below by constants times $V^{N-d(N)+\delta(N,d)}$.
\end{theorem}

This theorem differs from the previous in that we consider homotopy lens spaces up to isomorphism -- which necessitates an estimate on the size of the Reidemeister torsion (or equivalently, the torsion of an h-cobordism between two such manifolds).  We also need to modify our construction of examples (to get the lower bound); what we do is an algebraic K-theoretic analogue of the surgery theoretic construction used for Theorem~\ref{cor:AS}.

Finally, in Section~\ref{sec:density}, we will prove the following variation on~\cite{CW}. It is essentially independent of the rest of the paper.

\begin{theorem}\label{thm:CW}
Suppose $M$ is a closed oriented manifold of of dimension $4k+3$, where $k>0$, and $\pi_{1}(M)$ has finite subgroups of arbitrary large order. Then the values of $\rho^{(2)}(M’)$ as $M'$ varies over manifolds homotopy equivalent to $M$ is a dense subset of $\R$.  This implies that the group $S(M)$ of such manifolds (studied by surgery theory) is not finitely generated.
\end{theorem}

The condition on the fundamental group $\pi_1$ in this theorem arises in all known examples where $L^2$ Betti numbers do not form a discrete subset of $\R$. It seems that if the Farrell-Jones conjecture is true then this condition is necessary as well for the density of values of $\rho^{(2)}$. We remark that it is conjectured that all the numbers in this set of real numbers differ from one another by rational numbers~\cite{Chang04}. If this conjecture is true, it opens up the possibility of making estimates using p-adic heights of $\rho$-invariants.

We also remark that the conclusion of this theorem would follow from conjectures in~\cite{WY15}. In fact, it extends beyond the conjectures for certain groups, such as the lamplighter group, or more generally, finitely presented groups containing it, which consist solely of elements of order $2$ and also contain $\Z_2^n$ for all $n$.

We close by mentioning a drawback of the purely geometric nature of our methods. We are unable to prove the natural analog of Theorem~\ref{thm:AS} for Atiyah-Patodi-Singer invariants of odd dimensional PL manifolds with representations of their fundamental group. It seems unbelievable that this could not hold for Atiyah-Patodi-Singer invariants. But our methods that make essential use of bordism does not gracefully extend to the bordism of manifolds with flat bundles. This is due to the fact that the relevant bordism group is an infinite dimensional vector space rationally. Here, as in~\cite{CW},~\cite{Cha16}, and~\cite{Lim}, we avoid this in the case of $\rho^{(2)}$, using the functoriality of this invariant with respect to inclusions. Unfortunately, we do not yet have an analogous tool in the Atiyah-Patodi-Singer setting.

The paper is structured as follows: Section~\ref{sec:rho} provides a brief overview of the definitions of the Wall $\rho$-invariant and the Cheeger-Gromov $L^2$ $\rho$-invariants. In Section~\ref{sec:polyhedron}, we discuss the geometric concepts of $G$-colored polyhedra and towers of their coverings. We then employ these concepts to prove Theorem~\ref{thm:AS}. Section~\ref{sec:hyperbolization} is devoted to a quantitative study of relative hyperbolization, which is essential in proving Theorem~\ref{thm:CG}. In Section~\ref{sec:complexity}, we apply our results to the complexity and homotopy structure sets of homotopy lens spaces, providing proofs for Theorem~\ref{cor:CG} and Theorem~\ref{cor:AS}. Finally, in Section~\ref{sec:density}, we prove Theorem~\ref{thm:CW}, which is independent of the rest of the paper.

We would like to thank Jae Choon Cha and Fedya Manin for very stimulating conversations on a number of topics around the contents of this paper.

\noindent {\bf Note:} {\em In this paper, we assume that manifolds are closed, oriented, and triangulated unless otherwise noted.}

\section{Wall $\rho$-invariant and Cheeger-Gromov $L^2$ $\rho$-invariant}\label{sec:rho}

In this section, we provide a brief review of the definitions of the Wall $\rho$-invariant and the Cheeger-Gromov $L^2$ $\rho$-invariants.

The Atiyah-Singer’s classical $G$-signature~\cite[Section 6]{AS68} is a signature type invariant of odd-dimensional manifolds with finite fundamental group. This invariant is defined in the representation rings modulo the regular representations with rational coefficients. For a $2d$-manifold on which $G$ acts, the bilinear form on the middle dimension cohomology of the manifold is $G$-invariant. For even $d$, the operator which recognizes the bilinear form gives the positive eigenspace and negative eigenspace. So one have two real representations of $\rho^{+}$ and $\rho^{-}$ which are elements of the real representation ring $RO(G)$. The $G$-signature is defined as $\rho^{+} - \rho^{-}$. For odd $d$, we can change the skew symmetric inner product form into a symmetric Hermitian one by multiplying by $i$, and then using the construction given for $d$ even. So one can obtain a complex representation $\rho$ of $G$ and the $G$-signature is defined by $\rho-\rho^*$ in the complex representation ring $R(G)$. Recall $R(G)$ is the ring formally spanned by the isomorphism classes of complex representations of a group $G$ with $+$ and $\cdot$ corresponding to $\oplus$ and $\otimes$ respectively. Elements can be detected by their character $\chi_g(V)$ which is, by definition, the trace of $g$'s action on $V$ where $g \in G$. For even $d$ (odd $d$), the $G$-signature takes values in $\R$ ($i\R$). (For the trivial group, though, this gives only quadratic forms with signature $0$.)

With this background, we can define the $\rho$-invariant for any odd dimensional manifold endowed with a representation of its fundamental group.

\begin{definition}\label{def:ASrho}
\textbf{\textup{(Wall's classical $\rho$-invariant)}}
For a finite group $G$ and a $(2k-1)$-manifold $M$ endowed with a representation $\alpha : \pi_1(M) \to G$, the $\rho$-invariant $\rho(M,\alpha)$ is defined by:
\[
\rho(M, \alpha) :=\frac{1}{r}\sign_{G}(\Tilde{W}),
\]
where $W$ is a $2k$-manifold such that $\d W^{2k}=rM$, $\Tilde{W}$ is the induced $G$-cover of $W$, $\alpha$ factors through $\pi_1(W)$, and $\sign_{G}$ is the $G$-signature~\cite[Section 6]{AS68} of the induced representation $\alpha_{\ast}$ on the middle dimension cohomology of the induced $G$-cover $\Tilde{W}$ of $W$.
\end{definition}

\begin{remark}
The existence of the manifold $W$ and the well-definedness of the $\rho$-invariant are established using cobordism theory and the cohomological analogue of the G-index theorem, respectively~\cite{AS68}.
\end{remark}

Its significance extends to the classification of lens spaces~\cite{AB}, their homotopy analogues~\cite{CW21}, ~\cite{Wall99}, and its intimate connection with the $\eta$-invariant for the signature operator~\cite{APS}.

In \cite{CG} and \cite{CG85}, Cheeger and Gromov introduced $L^2$ analogues of the above invariants for manifolds with an infinite fundamental group. Specifically, they defined $L^2$ $\rho$-invariants on a closed $(4k-1)$-dimensional Riemannian manifold $M$. These invariants are the difference of the $\eta$-invariant of the signature operator of $M$ and the $L^2$ $\eta$-invariant of that of the $G$-cover of $M$, which is defined using the von Neumann trace. This analogue is generalized to the case of foliations in \cite{Ram}. Cheeger and Gromov also established the Cheeger-Gromov inequality, which gives universal bounds for the $\rho$-invariant. Chang and Weinberger extended the definition of the Cheeger-Gromov $\rho$-invariants to topological manifolds \cite{CW}.

\begin{definition}\label{def:CWrho}
\textbf{\textup{(Chang-Weinberger's topological definition of Cheeger-Gromov $\rho$-invariant)}} For a $(4k-1)$-manifold $M$, the Cheeger-Gromov $\rho$-invariant $\rho^{(2)}$ is defined by the $L^2$-signature defeat,
\[
\rho^{(2)}_{\Gamma}(M) := \frac{1}{r}(\sign_{\Gamma}^{(2)}W_{\Gamma} - \sign W) \in \R,
\]
where $W$ is a $4k$-manifold such that $\d W^{4k}=rM$ with $\pi = \pi_1(M)$ injecting into $\Gamma=\pi_1(W)$, $W_{\Gamma}$ is the induced $\Gamma$-cover of $W$, and $\sign_{\Gamma}^{(2)}$ is the $L^2$-signature of the symmetric form induced by cap product on the middle cohomology~\cite{CW}.
\end{definition}

\begin{remark}
The existence of $W$ is a consequence of Thom's classical work on cobordism and Hausmann~\cite{Hausmann}. The well-definedness of the $\rho$-invariant is obtained from the $\Gamma$-induction property of Cheeger-Gromov~\cite{CG85} and the usual Novikov additivity argument. For details, we refer readers to~\cite{CW},~\cite{Cha16}, and~\cite{Lim}.
\end{remark}

In~\cite{Cha16}, Cha established Cheeger-Gromov inequalities for topological manifolds of arbitrary dimension, and derived explicit universal linear bounds for $3$-manifolds in terms of the minimal complexity of triangulations. Subsequently, the first author provided more efficient bounds for 3-manifolds with controlled chain null-homotopies in all dimensions~\cite{Lim}.

Besides Cheeger and Gromov's original application to obtaining an integrality results for Hirzebruch $L$-classes of complete Riemannian manifolds with bounded curvature and finite volume, they have found numerous other applications; to showing that certain structure sets of manifolds are infinite~\cite{CW} (and indeed infinitely generated, as we will remark below), the moduli of metrics of positive scalar curvature have infinitely many components~\cite{PT}, to K-theory~\cite{LP01}, to three-dimensional topological knot cobordism~\cite{COT}, and to complexity of 3-manifolds (in various senses)~\cite{Cha16} among others.

\section{$G$-colored polyhedra and their coverings}\label{sec:polyhedron}

In this section, we prove Theorem~\ref{thm:AS} using the geometric concepts of G-colored polyhedra and their coverings.

We begin by observing that, for a $(2k-1)$-manifold $M$ endowed with a representation $\alpha : \pi_1(M) \to G$, the Wall $\rho$-invariant can be obtained from a cobordism $W$ between $M$ over $G$ and a manifold trivially over $G$\footnote{When we consider $M \to e$ as a $G$-space, it becomes the space $G \times M$ with the product action of $G$ on $G$, and $G$ acting trivially on $M$. We are giving here an equivariant cobordism between some number of copies of $\tilde{M}$ (the $G$-cover of $M$) and $G \times M$.)}, instead of a $2k$-manifold bounded by $M$. This is because the trivial end does not affect the $\rho$-invariant. To obtain a linear bound on the absolute value of the $\rho$-invariant, we note that the rank of the homology of the induced $G$-cover of $W$ in the middle dimension $k$ cannot be greater than $\lvert G \rvert$ times the number of $k$-cells in $W$. Thus, the dimension of the vector spaces arising in the definition of the $\rho$-invariant and therefore the absolute value of the $\rho$-invariant itself are also linearly bounded by the number of $k$-cells in $W$ by Definition~\ref{def:ASrho}. Therefore, our focus is on an explicit construction of $W$.

To do this we use the \emph{colored polyhedra and their coverings}, which are used to model chains in the Moore complex of a simplicial classifying space. 

We briefly recall the Moore complex of a simplicial classifying space. First, we define the Moore complex $\Z X_{\ast}$ of a simplicial set $X$. For details, we refer readers to excellent references~\cite{May} and~\cite{Moore}.

\begin{definition}\label{def:Moore}
Let $X$ be a simplicial set. The Moore complex $\Z X_{\ast}$ of $X$ is a chain complex of the abelian groups $\Z X_n$ endowed with the boundary maps $\d\colon\Z X_n \to \Z X_{n-1}$, where $n \in \{0,1,2,\cdots \}$. The group $\Z X_n$ is the free abelian group generated by the $n$-simplices of $X_n$. The boundary map $\d\colon\Z X_n \to \Z X_{n-1}$ is defined by $\d:=\Sigma_{i=0}^{n}(-1)^{i}d_i$, where $d_i$ is a face map of $X$.
\end{definition}

The Moore complex $\Z X_{\ast}$ of a simplicial set $X$ is not the same as the cellular chain complex $C_{\ast}(\lvert X  \rvert)$ of the geometric realization $\lvert X \rvert$ of $X$. Instead, there is a relation between $\Z X_{\ast}$ and $C_{\ast}(\lvert X  \rvert)$. For convenience, write $C_{\ast}(X) := C_{\ast}(\lvert X  \rvert)$. Let $D_{\ast}(X)$ be the subgroup of $\Z X_{\ast}$ generated by degenerate simplices of $X$. 

\begin{theorem}\label{thm:projection from the Moore complex}
\textup{(Mac Lane~\cite[p. 236]{ML}).} 
For a simplicial set $X$, there is a short exact sequence 
\[
0 \to D_{\ast}(X) \to \Z X_{\ast} \overset{p} \to C_{\ast}(X) \to 0 
\]
where the projection $p$ is a chain homotopy equivalence.
\end{theorem}
One can check
\[
C_{\ast}(X) \cong \frac{\Z X_{\ast}}{D_{\ast}(X)}.
\]

We can define the simplicial classifying space $BG$ of a discrete group $G$, giving a standard functorial simplicial construction of $BG$.

\begin{definition}
Let $G$ be a group. The {\em simplicial classifying space $BG$} of $G$ is defined to be a simplicial set with $BG_n = \{ [g_1, \ldots , g_n ] \mid  g_i \in G \}$ where $n \in \{0, 1, 2, \cdots \}$ together with  face maps $d_i : BG_n \to BG_{n-1}$ and  degeneracy maps $s_i : BG_n \to BG_{n+1}$  defined by
\begin{align*}
d_i[g_1,\ldots,g_n]&=
    \begin{cases}
      [g_2,\ldots,g_n] & i=0 \\
      [g_1,\ldots,g_i g_{i+1},\ldots,g_n] & 1 \leq i \leq n-1 \\
      [g_1,\ldots,g_{n-1}] & i=n \\
    \end{cases}\\
s_i[g_1, \ldots, g_n] &= [g_1, \ldots, g_i , e, g_{i+1}, \ldots, g_n]
\end{align*}
where $i = 0, 1, 2, \ldots, n$.
\end{definition}

Applying Definition~\ref{def:Moore} to the simplicial classifying space $BG$, we obtain the Moore complex of the simplicial classifying space $BG$.

\begin{definition}
Let $G$ be a group. The Moore complex $\Z BG_{\ast}$ of the simplicial classifying space $BG$ is a chain complex of free abelian groups $\Z BG_n$ which is generated by $n$-tuples $[g_1, \ldots , g_n ]$ of group elements $g_1, \ldots , g_n \in G$, together with the boundary map which is defined as the alternating sum of face maps $\d = \Sigma_{i=0}^{n}(-1)^i d_i$, where
\begin{center}
$
d_i[g_1,\ldots,g_n]=
    \begin{cases}
      [g_2,\ldots,g_n] & i=0 \\
      [g_1,\ldots,g_i g_{i+1},\ldots,g_n] & 1 \leq i \leq n-1 \\
      [g_1,\ldots,g_{n-1}] & i=n.
    \end{cases}
$
\end{center}
\end{definition}

We return now and consider how to construct a \emph{quantitative} controlled $2k$-chain in the Moore complex with an explicit and efficient complexity, whose boundary consists of copies of the fundamental class $[M]$ of $M$ in the classifying space of a finite group. This approach differs from the Cheeger-Gromov $\rho$-invariant case, where any group can be embedded into an acyclic one that we explained above. In Cheeger-Gromov situation, the acyclic group can even be constructed functorially (see~\cite{BDH}), but for finite groups, the $2k$-chain we need does not exist \emph{integrally}. Our new method gives a construction which requires multiplying our manifold by some number, but it is controlled by the triangulation of $M$ and the order of $G$. The method has the advantage of controlling the complexity of each step in the process. This allows us to obtain a cobordism $W$ between $M$ over $G$ and a manifold $N$ trivially over $G$ (i.e. a $G$-cobordism between $\Tilde{M}$ and $G \times N$), whose middle $k$-dimension homology satisfies a linear bound on its dimension, which is given by the number of $k$-cells in $W$.

We now define the concept of {\emph{a colored polyhedra and their covering}}.

\begin{definition}\label{def:polyhedron}
Let $G$ be a group. An \emph{(abstract) $G$-colored (directed) $n$-polyhedron} is an $n$-dimensional polyhedron satisfying the following properties:
\begin{enumerate}[label=(\roman*)]
\item Each edge is directed.
\item A group element is assigned to each edge.
\item Any path along edges represents a group element. To obtain the group element corresponding to a path, start with the identity element $e\in G$ and right-multiply the group element assigned to each edge in the order of the path. If the direction of the edge corresponds to the direction of the path, multiply by the assigned group element; otherwise, multiply by the inverse.
\item\label{loop} Any loop along edges represents a group relation. In other words, any two paths connecting two vertices represent the same group element.
\end{enumerate}
\end{definition}

\begin{remark}
By definition, each $(n-1)$-face of a $G$-colored $n$-polyhedron is also a $G$-colored $(n-1)$-polyhedron. Furthermore, we can obtain a new $n$-polyhedron $P$ by identifying two $G$-colored $n$-polyhedra $P_1$ and $P_2$ along the same $(n-1)$-face, and this new polyhedron $P$ is again a $G$-colored $n$-polyhedron.
\end{remark}

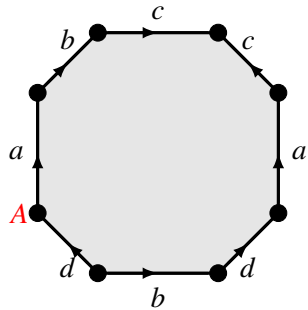
\begin{center}
\tikzset{->-/.style n args={2}{decoration={
  markings,
  mark=at position #1 with {\arrow[line width=1pt]{#2}}},postaction={decorate}}}
\begin{tikzpicture}[scale=0.8, bullet/.style={circle,inner sep=1.5pt,fill}]
 
 \begin{scope}  
  \path
   (-2,-1) node[red,label=above:](A){}
   (-2.3,-1) node[red,label=above:](){$A$}
   (-2,1) node[red,label=above:](B){}
   (-1,2) node[red,label=below:](C){}
   (1,2) node[red,label=above:](D){}
   (2,1) node[red,label=above:](E){}
   (2,-1) node[red,label=below:](F){}
   (1,-2) node[red,label=below:](G){}
   (-1,-2) node[red,label=below:](H){}
   ;
   {\draw[fill=black!10!,very thick] (-2,-1) -- (-2,1) -- (-1,2) -- (1,2) -- (2,1) -- (2,-1) -- (1,-2) -- (-1,-2) -- cycle;}
   
   {\draw [black,fill] (A) circle [radius=0.14];}
   {\draw [black,fill] (B) circle [radius=0.14];}
   {\draw [black,fill] (C) circle [radius=0.14];}
   {\draw [black,fill] (D) circle [radius=0.14];}
   {\draw [black,fill] (E) circle [radius=0.14];}
   {\draw [black,fill] (F) circle [radius=0.14];}
   {\draw [black,fill] (G) circle [radius=0.14];}
   {\draw [black,fill] (H) circle [radius=0.14];}
   ;
   {\draw[line width=1pt,->-={0.5}{latex}] (A) -- 
   node [text width=,midway,left=0.1em,align=center ] {$a$} (B);}
   {\draw[line width=1pt,->-={0.5}{latex}] (B) -- 
   node [text width=,midway,above=0.1em,align=center ] {$b$} (C);}
   {\draw[line width=1pt,->-={0.5}{latex}] (C) -- 
   node [text width=,midway,above=0.1em,align=center ] {$c$} (D);}
   {\draw[line width=1pt,->-={0.5}{latex}] (E) -- 
   node [text width=,midway,above=0.1em,align=center ] {$c$} (D);}
   {\draw[line width=1pt,->-={0.5}{latex}] (F) -- 
   node [text width=,midway,right=0.1em,align=center ] {$a$} (E);}
   {\draw[line width=1pt,->-={0.5}{latex}] (G) -- 
   node [text width=,midway,below=0.1em,align=center ] {$d$} (F);}
   {\draw[line width=1pt,->-={0.5}{latex}] (H) -- 
   node [text width=,midway,below=0.1em,align=center ] {$b$} (G);}
   {\draw[line width=1pt,->-={0.5}{latex}] (H) -- 
   node [text width=,midway,below=0.1em,align=center ] {$d$} (A);}
   \end{scope} 

\end{tikzpicture} 
\end{center}

As an example, consider the above $G$-colored $2$-polyhedron, where $G$ is a finite abelian group and $a, b, c, d \in G$. To check if this polyhedron satisfies~\ref{loop}, it is enough to verify if a loop starting from any vertex represents a group relation because $G$ is abelian. The clockwise loop starting from $A$ represents $abcc^{-1}a^{-1}d^{-1}b^{-1}d$, which is a group relation, indicating that the polyhedron is $G$-colored. 

We can interpret the $n$-simplices of the Moore complex $\Z BG_\ast$ \cite{Moore} of the simplicial classifying space $BG$ \cite{May} as $G$-colored $n$-polytopes. To be more precise, let us consider an oriented $n$-simplex determined by the vertices $h_0, h_1, \ldots, h_n \in G$, which we denote by the ordered tuple $(h_0, h_1, \ldots, h_n)$. Let $g_i = h_{i-1}^{-1}h_i$ for $i = 1, 2, \ldots, n$. We orient each edge of the simplex $(h_0,h_1,\ldots,h_n)$ joining $h_i$ and $h_j$ ($i<j$) from $h_i$ to $h_j$, and associate it with the group element $h_i^{-1}h_j ( = g_{i+1} g_{i+2} \cdots g_j)$. We can then represent this $G$-colored $n$-simplex by $[g_1, g_2, \ldots, g_n]$. In fact, $[g_1, g_2, \ldots, g_n]$ corresponds to the generator of the Moore complex $\Z BG_\ast$, and the orientation of the $n$-simplex determines the sign of this generator in $\Z BG_\ast$.

\begin{remark}
Please note that alternative notations for the $G$-colored $n$-simplex $[g_1, \ldots, g_n]$ may appear in other literature. In particular, one may encounter $(g_1, \ldots, g_n)$ or $[g_1 \vert \cdots \vert g_n]$ used interchangeably with $[g_1, \ldots, g_n]$.
\end{remark}

By virtue of~\ref{loop} in Definition~\ref{def:polyhedron}, a $G$-colored $n$-polyhedron can be decomposed into $G$-colored $n$-simplices, taking into account orientation. Moreover, any loop formed by edges of the simplices represents a group relation. In fact, the decomposition of a $G$-colored $n$-polyhedron corresponds to an $n$-chain in $\Z BG_\ast$. For instance, the decomposition of the aforementioned $G$-colored $2$-polyhedron corresponds to a $2$-chain in $\Z BG_2$, given by $[a,b]+[ab,c]-[ab,c]-[b,a]-[bd^{-1},d]+[d,bd^{-1}]$, where the choice of orientation is implicit.

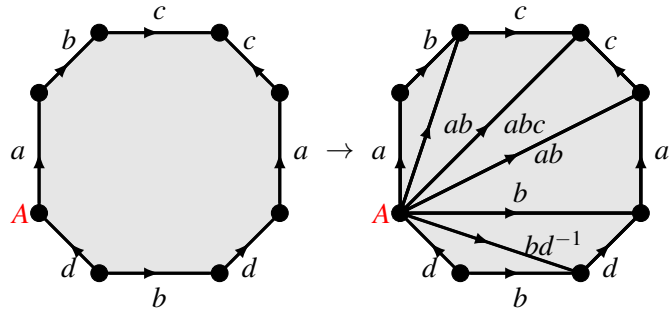
\begin{center}\label{figure:decomposition}
\tikzset{->-/.style n args={2}{decoration={
  markings,
  mark=at position #1 with {\arrow[line width=1pt]{#2}}},postaction={decorate}}}
\begin{tikzpicture}[scale=0.8, bullet/.style={circle,inner sep=1.5pt,fill}]
 
 \begin{scope}  
  \path
   (-2,-1) node[red,label=above:](A){}
   (-2.3,-1) node[red,label=above:](){$A$}
   (-2,1) node[red,label=above:](B){}
   (-1,2) node[red,label=below:](C){}
   (1,2) node[red,label=above:](D){}
   (2,1) node[red,label=above:](E){}
   (2,-1) node[red,label=below:](F){}
   (1,-2) node[red,label=below:](G){}
   (-1,-2) node[red,label=below:](H){}
   
   (3,0) node[label=above:](I){$\to$}
   ;
   {\draw[fill=black!10!,very thick] (-2,-1) -- (-2,1) -- (-1,2) -- (1,2) -- (2,1) -- (2,-1) -- (1,-2) -- (-1,-2) -- cycle;}
   
   {\draw [black,fill] (A) circle [radius=0.14];}
   {\draw [black,fill] (B) circle [radius=0.14];}
   {\draw [black,fill] (C) circle [radius=0.14];}
   {\draw [black,fill] (D) circle [radius=0.14];}
   {\draw [black,fill] (E) circle [radius=0.14];}
   {\draw [black,fill] (F) circle [radius=0.14];}
   {\draw [black,fill] (G) circle [radius=0.14];}
   {\draw [black,fill] (H) circle [radius=0.14];}
   ;
   {\draw[line width=1pt,->-={0.5}{latex}] (A) -- 
   node [text width=,midway,left=0.1em,align=center ] {$a$} (B);}
   {\draw[line width=1pt,->-={0.5}{latex}] (B) -- 
   node [text width=,midway,above=0.1em,align=center ] {$b$} (C);}
   {\draw[line width=1pt,->-={0.5}{latex}] (C) -- 
   node [text width=,midway,above=0.1em,align=center ] {$c$} (D);}
   {\draw[line width=1pt,->-={0.5}{latex}] (E) -- 
   node [text width=,midway,above=0.1em,align=center ] {$c$} (D);}
   {\draw[line width=1pt,->-={0.5}{latex}] (F) -- 
   node [text width=,midway,right=0.1em,align=center ] {$a$} (E);}
   {\draw[line width=1pt,->-={0.5}{latex}] (G) -- 
   node [text width=,midway,below=0.1em,align=center ] {$d$} (F);}
   {\draw[line width=1pt,->-={0.5}{latex}] (H) -- 
   node [text width=,midway,below=0.1em,align=center ] {$b$} (G);}
   {\draw[line width=1pt,->-={0.5}{latex}] (H) -- 
   node [text width=,midway,below=0.1em,align=center ] {$d$} (A);}
   \end{scope} 

 \begin{scope}[xshift=6cm]
  \path
   (-2,-1) node[red,label=above:](A){}
   (-2.3,-1) node[red,label=above:](){$A$}
   (-2,1) node[red,label=above:](B){}
   (-1,2) node[red,label=below:](C){}
   (1,2) node[red,label=above:](D){}
   (2,1) node[red,label=above:](E){}
   (2,-1) node[red,label=below:](F){}
   (1,-2) node[red,label=below:](G){}
   (-1,-2) node[red,label=below:](H){}
   ;
   {\draw[fill=black!10!,very thick] (-2,-1) -- (-2,1) -- (-1,2) -- (1,2) -- (2,1) -- (2,-1) -- (1,-2) -- (-1,-2) -- cycle;}
   {\draw[very thick] (-2,-1) -- (-2,1) -- (-1,2) -- cycle;}
   {\draw[very thick] (-2,-1) -- (-1,2) -- (1,2) -- cycle;}
   {\draw[very thick] (-2,-1) -- (1,2) -- (2,1) -- cycle;}
   {\draw[very thick] (-2,-1) -- (2,1) -- (2,-1) -- cycle;}
   {\draw[very thick] (-2,-1) -- (2,-1) -- (1,-2) -- cycle;}
   
   {\draw [black,fill] (A) circle [radius=0.14];}
   {\draw [black,fill] (B) circle [radius=0.14];}
   {\draw [black,fill] (C) circle [radius=0.14];}
   {\draw [black,fill] (D) circle [radius=0.14];}
   {\draw [black,fill] (E) circle [radius=0.14];}
   {\draw [black,fill] (F) circle [radius=0.14];}
   {\draw [black,fill] (G) circle [radius=0.14];}
   {\draw [black,fill] (H) circle [radius=0.14];}
   ;
   {\draw[line width=1pt,->-={0.5}{latex}] (A) -- 
   node [text width=,midway,left=0.1em,align=center ] {$a$} (B);}
   {\draw[line width=1pt,->-={0.5}{latex}] (B) -- 
   node [text width=,midway,above=0.1em,align=center ] {$b$} (C);}
   {\draw[line width=1pt,->-={0.5}{latex}] (C) -- 
   node [text width=,midway,above=0.1em,align=center ] {$c$} (D);}
   {\draw[line width=1pt,->-={0.5}{latex}] (E) -- 
   node [text width=,midway,above=0.1em,align=center ] {$c$} (D);}
   {\draw[line width=1pt,->-={0.5}{latex}] (F) -- 
   node [text width=,midway,right=0.1em,align=center ] {$a$} (E);}
   {\draw[line width=1pt,->-={0.5}{latex}] (G) -- 
   node [text width=,midway,below=0.1em,align=center ] {$d$} (F);}
   {\draw[line width=1pt,->-={0.5}{latex}] (H) -- 
   node [text width=,midway,below=0.1em,align=center ] {$b$} (G);}
   {\draw[line width=1pt,->-={0.5}{latex}] (H) -- 
   node [text width=,midway,below=0.1em,align=center ] {$d$} (A);}
   
   {\draw[line width=1pt,->-={0.5}{latex}] (A) -- 
   node [text width=,midway,right=0.1em,align=center ] {$ab$} (C);}
   {\draw[line width=1pt,->-={0.5}{latex}] (A) -- 
   node [text width=,midway,right=0.1em,align=center ] {$abc$} (D);}
   {\draw[line width=1pt,->-={0.5}{latex}] (A) -- 
   node [text width=,midway,right=0.1em,align=center ] {$ab$} (E);}
   {\draw[line width=1pt,->-={0.5}{latex}] (A) -- 
   node [text width=,midway,above=0.0em,align=center ] {$b$} (F);}
   {\draw[line width=1pt,->-={0.5}{latex}] (A) -- 
   node [text width=,midway,right=0.8em,align=center ] {$bd^{-1}$} (G);}
   \end{scope} 

\end{tikzpicture} 
\end{center}

As an abuse of notation, we use the symbol $P$ to refer to both a $G$-colored $n$-polyhedron and its corresponding decomposition or $n$-chain in $\Z BG_\ast$. Keeping in mind that $G$-colored polyhedra represent chains in $\Z BG_\ast$, we can define a concept that is analogous to a cycle in the chain complex.

\begin{definition}
A $G$-colored $n$-polyhedron $P$ is called \emph{a cycle} if it represents a cycle in the chain complex $\Z BG_\ast$, i.e., $\partial P = 0$.
\end{definition}

Indeed, the $G$-colored $2$-polyhedron $P$ presented earlier is a cycle, as can be easily verified by computing its boundary: $\d([a,b]+[ab,c]-[ab,c]-[b,a]-[bd^{-1},d]+[d,bd^{-1}])=0$.

Next, we introduce a $G$-colored $n$-polyhedron endowed with \emph{a vertex set of group elements}. 

\begin{definition}
Let $P$ be a $G$-colored $n$-polyhedron. we may assign a group element to a vertex. This assignment determines the group element at every other vertex of $P$ by right multiplication of the group elements of the edges, according to the direction of the edges. Then, $P$ is said to be \emph{endowed with a vertex set of group elements}.
\end{definition}

For example, if we assign $e$ to the vertex $A$ in the previous example, then we obtain below:

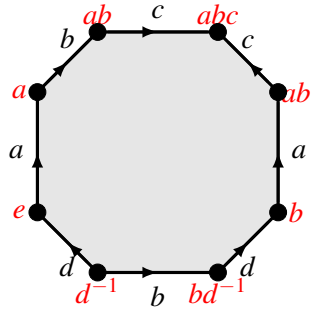
\begin{center}
\tikzset{->-/.style n args={2}{decoration={
  markings,
  mark=at position #1 with {\arrow[line width=1pt]{#2}}},postaction={decorate}}}
\begin{tikzpicture}[scale=0.8, bullet/.style={circle,inner sep=1.5pt,fill}]
 
 \begin{scope}  
  \path
   (-2,-1) node[black,label=above:](A){}
   (-2.3,-1) node[red,label=above:](){$e$}
   (-2,1) node[black,label=above:](B){}
   (-2.3,1) node[red,label=above:](){$a$}
   (-1,2) node[black,label=below:](C){}
   (-1,2.3) node[red,label=above:](){$ab$}
   (1,2) node[black,label=above:](D){}
   (1,2.3) node[red,label=above:](){$abc$}
   (2,1) node[black,label=above:](E){}
   (2.3,1) node[red,label=above:](){$ab$}
   (2,-1) node[black,label=below:](F){}
   (2.3,-1) node[red,label=above:](){$b$}
   (1,-2) node[black,label=below:](G){}
   (1,-2.3) node[red,label=above:](){$bd^{-1}$}
   (-1,-2) node[black,label=below:](H){}
   (-1,-2.3) node[red,label=above:](){$d^{-1}$}
   ;
   {\draw[fill=black!10!,very thick] (-2,-1) -- (-2,1) -- (-1,2) -- (1,2) -- (2,1) -- (2,-1) -- (1,-2) -- (-1,-2) -- cycle;}
   
   {\draw [black,fill] (A) circle [radius=0.14];}
   {\draw [black,fill] (B) circle [radius=0.14];}
   {\draw [black,fill] (C) circle [radius=0.14];}
   {\draw [black,fill] (D) circle [radius=0.14];}
   {\draw [black,fill] (E) circle [radius=0.14];}
   {\draw [black,fill] (F) circle [radius=0.14];}
   {\draw [black,fill] (G) circle [radius=0.14];}
   {\draw [black,fill] (H) circle [radius=0.14];}
   ;
   {\draw[line width=1pt,->-={0.5}{latex}] (A) -- 
   node [text width=,midway,left=0.1em,align=center ] {$a$} (B);}
   {\draw[line width=1pt,->-={0.5}{latex}] (B) -- 
   node [text width=,midway,above=0.1em,align=center ] {$b$} (C);}
   {\draw[line width=1pt,->-={0.5}{latex}] (C) -- 
   node [text width=,midway,above=0.1em,align=center ] {$c$} (D);}
   {\draw[line width=1pt,->-={0.5}{latex}] (E) -- 
   node [text width=,midway,above=0.1em,align=center ] {$c$} (D);}
   {\draw[line width=1pt,->-={0.5}{latex}] (F) -- 
   node [text width=,midway,right=0.1em,align=center ] {$a$} (E);}
   {\draw[line width=1pt,->-={0.5}{latex}] (G) -- 
   node [text width=,midway,below=0.1em,align=center ] {$d$} (F);}
   {\draw[line width=1pt,->-={0.5}{latex}] (H) -- 
   node [text width=,midway,below=0.1em,align=center ] {$b$} (G);}
   {\draw[line width=1pt,->-={0.5}{latex}] (H) -- 
   node [text width=,midway,below=0.1em,align=center ] {$d$} (A);}
   \end{scope} 

\end{tikzpicture} 
\end{center}

\begin{remark}\label{rmk:simplicial cylinder}
For a finite group $G$, there are exactly $\lvert G \rvert$ ways to assign group elements to vertices of a $G$-colored $n$-polyhedron. 

Given a $G$-colored $n$-polyhedron $P = \Sigma_{i=1}^{k} \sigma_i^n$ endowed with a vertex set $T$, we can construct a \emph{simplicial cylinder between $P$ and $E$ with respect to $T$}~~\cite[Section 4.3]{Lim}, denoted by $\textup{Cyl}(P,E,T)$. Here, $\sigma^n_i$ is a signed $n$-simplex, $E=\Sigma_{i=1}^{k} \textup{sign}(\sigma_i^n)[e,e,\cdots,e]$ is a degenerate $G$-colored $n$-chain, and $\textup{sign}(\sigma_i^n)$ is the sign of the $i$-th simplex in $P$.

The simplicial cylinder $\textup{Cyl}(P,E,T)$ is a $G$-colored $(n+1)$-chain in $\Z BG_{\ast}$. Intuitively, it can be thought of as a prism whose top is $P$, whose base is $E$, and whose edges connect corresponding vertices in $P$ and $E$ are $T$.
\end{remark}

In fact, the previous example can be converted to a simplicial cylinder as modeled below;

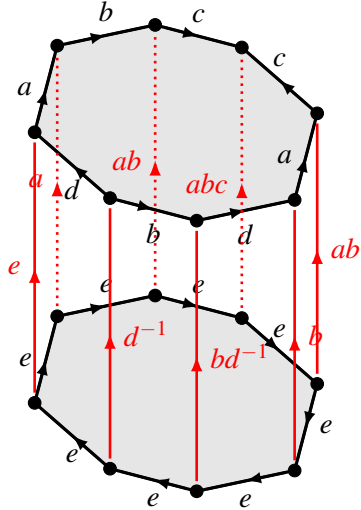
\begin{center}
\tikzset{->-/.style n args={2}{decoration={
  markings,
  mark=at position #1 with {\arrow[line width=1pt]{#2}}},postaction={decorate}}}
\begin{tikzpicture}[scale=0.6, bullet/.style={circle,inner sep=1.5pt,fill}]
 
   \begin{scope}
  \path
   (0,0) node[black,label=above:](A){}
   (0.5,1.92) node[black,label=above:](B){}
   (2.67,2.38) node[black,label=below:](C){}
   (4.59,1.88) node[black,label=above:](D){}
   (6.26,0.42) node[black,label=above:](E){}
   (5.76,-1.5) node[black,label=below:](F){}
   (3.59,-1.96) node[black,label=below:](G){}
   (1.67,-1.46) node[black,label=below:](H){}
   
   (0,-6) node[black,label=above:](I){}
   (0.5,1.92-6) node[black,label=above:](J){}
   (2.67,2.38-6) node[black,label=below:](K){}
   (4.59,1.88-6) node[black,label=above:](L){}
   (6.26,0.42-6) node[black,label=above:](M){}
   (5.76,-1.5-6) node[black,label=below:](N){}
   (3.59,-1.96-6) node[black,label=below:](O){}
   (1.67,-1.46-6) node[black,label=below:](P){}
   
   
   ;
   {\draw[fill=black!10!,very thick] (0,0) -- (0.5,1.92) -- (2.67,2.38) -- (4.59,1.88) -- (6.26,0.42) -- (5.76,-1.5) -- (3.59,-1.96) -- (1.67,-1.46) -- cycle;}
   {\draw[fill=black!10!,very thick] (0,-6) -- (0.5,1.92-6) -- (2.67,2.38-6) -- (4.59,1.88-6) -- (6.26,0.42-6) -- (5.76,-1.5-6) -- (3.59,-1.96-6) -- (1.67,-1.46-6) -- cycle;}
   
   {\draw [black,fill] (A) circle [radius=0.14];}
   {\draw [black,fill] (B) circle [radius=0.14];}
   {\draw [black,fill] (C) circle [radius=0.14];}
   {\draw [black,fill] (D) circle [radius=0.14];}
   {\draw [black,fill] (E) circle [radius=0.14];}
   {\draw [black,fill] (F) circle [radius=0.14];}
   {\draw [black,fill] (G) circle [radius=0.14];}
   {\draw [black,fill] (H) circle [radius=0.14];}
   
   {\draw [black,fill] (I) circle [radius=0.14];}
   {\draw [black,fill] (J) circle [radius=0.14];}
   {\draw [black,fill] (K) circle [radius=0.14];}
   {\draw [black,fill] (L) circle [radius=0.14];}
   {\draw [black,fill] (M) circle [radius=0.14];}
   {\draw [black,fill] (N) circle [radius=0.14];}
   {\draw [black,fill] (O) circle [radius=0.14];}
   {\draw [black,fill] (P) circle [radius=0.14];}
   ;
   {\draw[line width=1pt,->-={0.5}{latex}] (A) -- 
   node [text width=,midway,left=0.1em,align=center ] {$a$} (B);}
   {\draw[line width=1pt,->-={0.5}{latex}] (B) -- 
   node [text width=,midway,above=0.1em,align=center ] {$b$} (C);}
   {\draw[line width=1pt,->-={0.5}{latex}] (C) -- 
   node [text width=,midway,above=0.1em,align=center ] {$c$} (D);}
   {\draw[line width=1pt,->-={0.5}{latex}] (E) -- 
   node [text width=,midway,above=0.1em,align=center ] {$c$} (D);}
   {\draw[line width=1pt,->-={0.5}{latex}] (F) -- 
   node [text width=,midway,left=0.1em,align=center ] {$a$} (E);}
   {\draw[line width=1pt,->-={0.5}{latex}] (G) -- 
   node [text width=,midway,below=0.1em,align=center ] {$d$} (F);}
   {\draw[line width=1pt,->-={0.5}{latex}] (H) -- 
   node [text width=,midway,below=0.1em,align=center ] {$b$} (G);}
   {\draw[line width=1pt,->-={0.5}{latex}] (H) -- 
   node [text width=,midway,below=0.1em,align=center ] {$d$} (A);}
   
   {\draw[line width=1pt,->-={0.5}{latex}] (I) -- 
   node [text width=,midway,left=0.1em,align=center ] {$e$} (J);}
   {\draw[line width=1pt,->-={0.5}{latex}] (J) -- 
   node [text width=,midway,above=0.1em,align=center ] {$e$} (K);}
   {\draw[line width=1pt,->-={0.5}{latex}] (K) -- 
   node [text width=,midway,above=0.1em,align=center ] {$e$} (L);}
   {\draw[line width=1pt,->-={0.5}{latex}] (L) -- 
   node [text width=,midway,above=0.1em,align=center ] {$e$} (M);}
   {\draw[line width=1pt,->-={0.5}{latex}] (M) -- 
   node [text width=,midway,right=0.1em,align=center ] {$e$} (N);}
   {\draw[line width=1pt,->-={0.5}{latex}] (N) -- 
   node [text width=,midway,below=0.1em,align=center ] {$e$} (O);}
   {\draw[line width=1pt,->-={0.5}{latex}] (O) -- 
   node [text width=,midway,below=0.1em,align=center ] {$e$} (P);}
   {\draw[line width=1pt,->-={0.5}{latex}] (P) -- 
   node [text width=,midway,below=0.1em,align=center ] {$e$} (I);}
   
   {\draw[red,line width=1pt,->-={0.5}{latex}] (I) -- 
   node [red,text width=,midway,left=0.1em,align=center ] {$e$} (A);}
   {\draw[red,dotted,line width=1pt,->-={0.5}{latex}] (J) -- 
   node [red,text width=,midway,left=0.1em,align=center ] {$a$} (B);}
   {\draw[red,dotted,line width=1pt,->-={0.5}{latex}] (K) -- 
   node [red,text width=,midway,left=0.1em,align=center ] {$ab$} (C);}
   {\draw[red,dotted,line width=1pt,->-={0.5}{latex}] (L) -- 
   node [red,text width=,midway,left=0.1em,align=center ] {$abc$} (D);}
   {\draw[red,line width=1pt,->-={0.5}{latex}] (M) -- 
   node [red,text width=,midway,right=0.1em,align=center ] {$ab$} (E);}
   {\draw[red,line width=1pt,->-={0.5}{latex}] (N) -- 
   node [red,text width=,midway,right=0.1em,align=center ] {$b$} (F);}
   {\draw[red,line width=1pt,->-={0.5}{latex}] (O) -- 
   node [red,text width=,midway,right=0.1em,align=center ] {$bd^{-1}$} (G);}
   {\draw[red,line width=1pt,->-={0.5}{latex}] (P) -- 
   node [red,text width=,midway,right=0.1em,align=center ] {$d^{-1}$} (H);}
   \end{scope} 
\end{tikzpicture} 
\end{center}

We now introduce the concept of a \emph{covering} of a cycle.

\begin{definition}\label{def:covering}
Let $G$ be a finite group and $P$ be a (simplex-decomposed) $G$-colored $n$-cycle. Assume that there exists a $(n-1)$-simplex $B$ in $\partial P$ (so there exists $-B$ in $\partial P$ as well). A \emph{covering of $P$ with respect to $B$} is a $G$-colored $n$-polyhedron obtained by gluing $\lvert G \rvert$ copies of $P$ in the following way: $-B$ on the $i$-th copy of $P$ is glued to $B$ on the $(i+1)$-th copy of $P$, where $i=1,2,\ldots,\lvert G\rvert-1$.
\end{definition}

\begin{remark}\label{rmk:the same vertex1}
Note that a covering of a cycle, constructed as described in Definition \ref{def:covering}, is again a cycle. Moreover, since the covering is a $G$-colored $n$-polyhedron endowed with a vertex set of group elements, and $G$ is finite, the boundary of the covering contains $+B$ and $-B$ with the same group element assigned to their vertices. This allows us to \emph{geometrically} identify $+B$ and $-B$ on the boundary of a covering of $P$.
\end{remark}

We can construct a tower of coverings inductively.

\begin{definition}\label{def:tower}
Let $G$ be a finite group and let $P$ be a simplex-decomposed $G$-colored $n$-cycle. Let $\pm B_1$ and $\pm B_2$ be $(n-1)$-simplices on $\d P$ such that $B_1 \neq B_2$. Let $P_{B_1}$ be a covering of $P$ with respect to $B_1 \in \d P$. A \emph{tower of coverings of $P$ with respect to $B_1$ and $B_2$ in order}, denoted by $P_{(B_1,B_2)}$, is a $G$-colored $n$-polyhedron obtained by gluing $(n-1)$-simplices of $\pm B_2$ of $P_{B_1}$ in the following way: Since there are exactly $\lvert G \rvert$ ways to assign group elements to the vertex set of $P$. Then there are $\lvert G \rvert$ pairs of $+B_2$ and $-B_2$ with the same group elements at their vertices. We glue the cancelled pairs of $+B_2$ and $-B_2$ except only one pair. Inductively we can define a tower of covering of $P$ with respect to $B_1, B_2, \cdots, B_s \in \d P$ from $\lvert G\rvert$ copies of $P$.
\end{definition}

\begin{remark}
Since we construct a tower of coverings of $P$ in a way to labeling of the vertices consistent, the tower of coverings is a $G$-colored $n$-polyhedron.
\end{remark}

\begin{remark}\label{rmk:the same vertex2}
Since $G$ is a finite group, the vertices of the corresponding pair of simplices $+B_k$ and $-B_k$ in $\partial P_{(B_1,B_2)}$ $(k=1,2)$ must have the same group elements, for any choice of vertex sets. See Remark~\ref{rmk:the same vertex1}.
\end{remark}

We now present a lemma that plays a crucial role in obtaining a linear bound for the complexity of $u$ in the step~(\ref{step:2}).

\begin{lemma}\label{lem:n+1 chain}
Let $G$ be a finite group and $C=\sum_{i=1}^k \sigma_i^n$ be an $n$-cycle in $\mathbb{Z}BG_n$. Then there exists an $(n+1)$-chain $u$ whose boundary is $N_{n,G}$ copies of $C-\Sigma_{i=1}^{k}\textup{sign}(\sigma^n_i) \cdot [e, e,\cdots, e]$ and whose simplicial complexity is $N_{n,G} \cdot \lvert C \rvert$, where $\sigma^n_i$ is a signed $n$-simplex in $\Z BG_n$ and $N_{n,G}$ is a natural number depending on the dimension $n$ and the order of $G$.
\end{lemma}

\begin{proof}
Recall that all faces of $\sigma^n_i$ ($i=1, 2, \dots, k$) are algebraically paired up because $C$ is an $n$-cycle. To construct a set of $G$-colored $n$-polytopes that represents $C \in \Z BG_n$, we begin by picking a signed $n$-simplex $\sigma^n_1$. If all of its faces are already paired up algebraically, we keep $\sigma^n_1$ and move on to the next simplex. Otherwise, we glue $\sigma^n_1$ to another simplex along the faces that are algebraically paired up. We continue this process, gluing $n$-simplices to the polytope, until we cannot pair up any more faces. Since there are finitely many $n$-simplices, this process must terminate and give us a set of $G$-colored $n$-polytopes whose union represents $C$. We denote this set of $G$-colored $n$-polytopes by $\{P_1, P_2, \dots, P_l \}$, where $l \leq k$. By construction, each $P_j$ is a cycle for all $j \in \{1, 2, \dots, l\}$.

Since each $G$-colored $n$-polytope $P_j$ is a cycle, we can pair up the $(n-1)$-simplices in $\partial P_j$ as $+B_k$ and $-B_k$, where $k = 1, 2, \dots, |\partial P_j|/2$. We can then recursively construct a tower of covering polytopes $P_{(B_1, B_2, \dots, B_{|\partial P_j|/2})}$ of $P$ with respect to $B_1, B_2, \dots, B_{|\partial P_j|/2}$, denoted briefly by $\overline{P_j}$.

Put $\overline{P_j}=\Sigma_{m=1}^{s} \sigma^j_m$. Then, as noted in Remark~\ref{rmk:simplicial cylinder}, we can construct an $(n+1)$-chain $\textup{Cyl}(\overline{P_j},\overline{E_j},T)$ by endowing $\overline{P_j}$ with a vertex set. We choose $\overline{E_j}$ to be $\Sigma_{m=1}^{s} \textup{sign}(\sigma^j_m)[e,e,\cdots,e]$, and $T$ to be the corresponding vertex set.

We apply~\cite[Lemma 4.7]{Lim}, which implies that $\partial \operatorname{Cyl}(\overline{P_j},\overline{E_j},T) = \overline{P_j} - \overline{E_j} - \Sigma \operatorname{Cyl}(\partial \overline{P_j}, \partial \overline{E_j}, \partial T)$. Using Remark~\ref{rmk:the same vertex1} and Remark~\ref{rmk:the same vertex2}, we note that algebraically paired $(n-1)$-simplices $+B$ and $-B$ \-in $\partial \overline{P_j}$ have the same group element for any vertex set. Therefore, the simplicial cylinders of $+B$ and $-B$ are the same, except for the sign. As a result, $\Sigma \operatorname{Cyl}(\partial \overline{P_j}, \partial \overline{E_j}, \partial T)=0$ for all $j$. This means that $\Sigma_{j=1}^l \operatorname{Cyl}(\overline{P_j},\overline{E_j},T)$ is the desired $(n+1)$-chain $u$, whose boundary is some copies of $C-\Sigma_{i=1}^{k}\operatorname{sign}(\sigma^n_i) \cdot [e, e,\cdots, e]$.

Put $P=\Sigma_j \overline{P_j}$. Using the constructions of Definitions~\ref{def:covering} and \ref{def:tower}, we obtain $\lvert P \rvert \leq \lvert G \rvert \cdot \lvert C \rvert$. Moreover, according to Remark~\ref{rmk:simplicial cylinder}, we have $\lvert u \rvert \leq (n+1) \cdot \lvert G \rvert \cdot \lvert C \rvert$.
\end{proof}

Now we prove Theorem~\ref{thm:AS}.

\begin{proof}[Proof of Theorem~\ref{thm:AS}]

Let $\pi$ be the fundamental group of $M$. Here, $M$ refers to both the triangulated $(2k-1)$-manifold $M$ and the geometric realization $\lvert M \rvert$ of the simplicial set induced by the triangulation.

Our proof builds on ideas of~\cite{Cha16} and~\cite{CL}. Specifically, we proceed as follows with some details to be filled in after:

\begin{enumerate}
    
\item\label{step:1} We begin by utilizing the simplicial-cellular approximation of maps to a classifying space \cite[Theorem 3.7]{Cha16}. This yields that a chain map $\alpha_\ast: C_\ast(M) \rightarrow \Z BG_\ast$ is induced by the given representation $\alpha:\pi_1(M) \rightarrow G$, where $\Z BG_\ast$ denotes the Moore complex of the simplicial classifying space $BG$. We embed the chain $[M] \in C_{2k}(M)$ representing the fundamental class of $M$ into $\Z BG_\ast$.

\item\label{step:2} As $G$ is a finite group and $[M]$ is a cycle, we may affirm that there exists a $2k$-chain $u$ whose boundary is $r\alpha_{\ast}([M])$ for some $r \in \mathbb{N}$. To construct such a chain $u$ explicitly, we use the \emph{colored polyhedra and their covering} in the Moore complex $\mathbb{Z}BG_\ast$ of the simplicial classifying space $BG$ introduced above.

\item\label{step:3} We employ a result due to Theorem~\ref{thm:projection from the Moore complex}~\cite[p. 236]{ML}, which asserts the existence of a projection map $p : \Z BG_\ast \to C_\ast(BG)$. In particular, for any generator of $\Z BG_\ast$, the image under $p$ is a chain consisting of at most one simplex in the cellular chain complex $C_\ast(BG)$. With this in mind, we map the $2k$-chain $u$ to its image $p(u)$ in $C_\ast(BG)$.

\item\label{step:4} By the simplicial-cellular approximation of maps to a classifying space \cite[Theorem 3.7]{Cha16}, we may assume that $M$ is over $BG^{(2k-1)}$, where $BG^{(2k-1)}$ is the $(2k-1)$-skeleton of the geometric realization of $BG$. By combining our $2k$-chain $u$ in the step~\ref{step:2} and a high-dimensional analogue of the geometric construction in~\cite[Proposition 3.10]{Cha16}, we can construct a cobordism $W_{2k-1}$ between $rM$ over $BG^{(2k-1)}$ and $N_{2k-2}$ over $BG^{(2k-2)}$, adding $1$-handles between algebraically corresponding pairs of $(2k-1)$-simplices in $\d p(u)$ and $p(r\alpha_\ast([M]))$ 

\item\label{step:5} As a backward-inductive process, we construct a sequence of cobordisms $W_{i}$ over $G$ from $N_i \to BG^{(i)}$ to $N_{i-1} \to BG^{(i-1)}$, where $N_{i}$ is a $(2k-1)$-manifold for $i=2k-2, 2k-3, \cdots, 2, 1$. To do this, we combine our covering of colored polyhedra with a geometric idea from~\cite{CL}.

\item\label{step:6} Gluing these cobordisms $W_i$ together, we obtain a cobordism $W$ over $G$ from $rM \to BG^{(2k-1)}$ to $N_0 \to BG^{(0)}$. As $BG$ is connected, the map $N_0 \to BG^{(0)}$ is homotopic to a constant map. By modifying the map $W\to BG$ on a collar neighborhood of $N_0$ using a homotopy, we obtain a desired cobordism $W$ between the given $rM$ and a trivial end $N_0$.

\end{enumerate}

Now for the necessary details we briefly recall the \emph{weight} of a chain map introduced in~\cite{Cha16}. For a positive based chain complex over $\Z$, the {\em mass} $d(u)$ of a chain $u=\Sigma_{\alpha} n_\alpha e_\alpha$ is defined by $d(u):=\Sigma_{\alpha} \vert n_\alpha \vert$, which is the $L^1$-norm. Suppose $C_\ast$ and $D_\ast$ are based chain complexes. For a chain map $f \colon C_\ast \to D_\ast$, the {\em mass function $d_f$ of the chain map} is defined by $d_f(k) := max \{d(f(c)) \mid c \in C_i$ is a basis element, $i \leq k \}$. In this case, we say the chain map $f$ is controlled by the mass function $d_f(k)$. We can similarly define the weight of a chain homotopy.

We verify that the step~(\ref{step:1}) does not increase complexity. We embed the chain $[M] \in C_{\ast}(M)$ representing the fundamental class of $M$ into $\Z BG_{\ast}$. As a cellular complex and a simplicial complex, the triangulation of $M$ induces both the cellular chain complex $C_{\ast}(M)$ and the Moore complex $\Z_{\ast}(M)$. Furthermore, there exists a canonical inclusion $i\colon C_\ast(M) \to \Z_\ast(M)$. According to the simplicial-cellular approximation theorem~\cite[Theorem 3.7]{Cha16}, there exists a chain map $\alpha_{\ast} : \Z (M) \to \Z BG$ induced by the given representation $\alpha : \pi_1(M) \to G$. Abusing notation, we denote the composition of chain maps 
\[
C_\ast(M) \overset{i} \to \Z_\ast(M) \overset{\alpha_{\ast}} \to \Z BG_\ast
\]
by $\alpha_{\ast}$. So we can embed $[M]$ into $\Z BG_{\ast}$.

In the step~(\ref{step:2}), we construct a $2k$-chain $u\in C_\ast(BG)$ whose boundary is some copies of $p(\alpha_{\ast}([M]))$ and whose mass linearly depends on the complexity of $M$ with a multiplicative coefficient depending only on the dimension $n=2k-1$ and $\lvert G \rvert$. To construct such a chain $u$, we apply Lemma~\ref{lem:n+1 chain} to a $n$-cycle $C=i([M]) \in \Z_{n}(M)$. Let $\Phi(C) \in \Z BG_{n+1}$ be the resulting $(n+1)$-chain. Specifically, we have $\Delta(\Phi(C)) \leq N_{n,G} \cdot \Delta(M)$ and $\partial \Phi(C) = N_{n,G} \cdot (C-\sum_i\mathrm{sign}(\sigma_i^n)[e,e,\ldots,e])$, where $N_{n,G}=(n+1) \cdot \lvert G \rvert$.

We show that the step~(\ref{step:3}) does not elevate complexity. According to~\cite{ML}, there exists a projection map
\[
p : \Z BG_{\ast} \to C_{\ast}(BG) \cong \Z BG_{\ast}/D_{\ast}(BG),
\]
where $D_{\ast}(X)$ is the subgroup of $\Z X_{\ast}$ generated by degenerate simplices of $X$. Note the projection are controlled chain maps by the constant function of $1$. In other words, the weight of composition of chain maps
\[
C_\ast(M) \overset{i} \to \Z_\ast(M) \overset{\alpha_{\ast}} \to \Z BG_\ast \overset{p} \to C_\ast(BG)
\]
is controlled by the constant function of $1$. We then project the $(n+1)$-chain $\Phi(C)$ to $C_{2k}(BG)$ using $p :\Z BG_\ast \to C_\ast(BG)$, which yields the desired chain $u = p(\Phi(i([M]))) \in C_{2k}(BG)$. By Lemma~\ref{lem:n+1 chain}, $\partial u = N_{n,G} \cdot p(\alpha_\ast(i_\ast([M])))$. Furthermore, we have $\lvert u \rvert \leq N_{n,G} \cdot \Delta(M)$.

Step~(\ref{step:4}) contributes only a linear bound for the complexity if $W$. In this step~(\ref{step:4}), we construct a cobordism $W_{2k-1}$ over $BG$ between $M \to BG^{(2k-1)}$ induced by the given representation of $\pi(M) \to G$ and $M_{d-1} \to BG^{(2k-2)}$. Consider the $\lvert G \rvert$ copies of $M \times [0,1]$. Notice that, by our construction of the covering of colored polyhedra, we obtain the $2k$-chain $u$ such that $\d u = \lvert G \rvert \cdot \alpha_{\ast}([M])$. For a $2k$-cell $e^{2k}_\alpha$ of $BG$, we say the characteristic map is given by $\varphi_\alpha\colon D^{2k}_\alpha\to BG^{(2k)}$, where $D^{2k}_\alpha$ is a $2k$-disk. One may assume that the center of each $2k$-cell of $BG$ is a regular value of $\varphi\colon M\to BG^{(2k)}$ and a regular value of each attaching map $\varphi_\alpha|_{\partial D^{2k}_\alpha}\colon \partial D^{2k}_\alpha \to BG^{(2k)}$. Given the $2k$-chain $u=-\sum_\alpha n_\alpha^{\vphantom{2k}}e_\alpha^{2k}$, we first consider the disjoint union $\bigsqcup_{\lvert G \rvert} (M\times[0,1]) \sqcup \bigsqcup_\alpha n^{\vphantom{1}}_\alpha D^{2k}_\alpha$, together with the projection $M\times[0,1]\to M$ and the maps~$\varphi_\alpha$. The relation $\d u = \lvert G \rvert \cdot \alpha_{\ast}([M])$ gives that for the center $y$ of each $d$-cell of $BG$, the points in the preimage of $y$ in $\bigsqcup_{\lvert G \rvert} M \times \{1\}$ and $\bigsqcup_\alpha n^{\vphantom{1}}_\alpha \d D^{d+1}_\alpha$ signed by the local degree are cancelled in pairs. For each cancelling pair, we can attach a $1$-handle joining these to the disjoint union with a corresponding orientation with respect to $M$ since the attaching $0$-sphere can be framed by pulling back a fixed framing at the regular value~$y$. We denote the resulting oriented cobordism by $W_{2k-1}$. Notice that the image of the new boundary $M_{2k-2}$ is disjoint from the centers of $(2k-1)$-cells in~$BG^{(2k-1)}$. Using a homotopy on a collar neighborhood, one can obtain that $M_{2k-2}$ is over~$BG^{(2k-2)}$. The number of $1$-handles depends linearly on the complexity of $u$ and $M$. Also, $c(u)$ linearly depends on $c(M)$ as we observed in the step~(\ref{step:2}). The subdivision obtaining a triangulation of $M_{d-1}$ results in a complexity which is linearly dependent on $c(M)$. This step contributes only a multiplicative coefficient that depends only on the dimension $2k-1$.

Now we explain step~(\ref{step:5}) and why it contributes only a linear increase in the complexity of $W$. From the previous step~(\ref{step:4}), we have a cobordism $W_{2k-1}$ over $BG$ between $M \to BG^{(2k-1)}$ induced by the given representation of $\pi(M) \to G$ and $M_{2k-2} \to BG^{(2k-2)}$. From the construction, the bottom map is actually $M_{2k-2} \to u^{(2k-2)} \subset BG^{(2k-2)}$. Since we explicitly constructed $u$ as a $G$-colored polyhedron endowed with a vertex set of group elements, there exists a simplicial cylinder structure from $u$ to $E$, where $E$ is the corresponding degenerate $2k$-chain as described in Remark~\ref{rmk:simplicial cylinder}. Then, for each $i$-simplex $\sigma$ $i=2k-2, 2k-1, \cdots 2, 1, 0$ in $u^{(2k-2)}$, there exists an $(i+1)$-chain $\Psi(\sigma)$ such that $\d \Psi(\sigma) + \Psi(\d \sigma) = \sigma - [e,e,\cdots,e]$ where $[e,e,\cdots,e]$ is the degenerate $i$-simplex. By PL-transversality of the simplicial-cellular map $\varphi : M_{2k-2} \to u^{(2k-2)} \subset BG^{(2k-2)}$, we can assume that $\varphi^{-1}(c_\sigma)$ is a closed manifold and $\varphi^{-1}(b_\tau)$ is a manifold with boundary such that
\[
\Sigma_{\sigma \in u^{(2k-2)}} \sigma \times \varphi^{-1}(c_\sigma) = \Sigma_{\sigma \in u^{(2k-2)}} \d \Psi (\sigma) \times \varphi^{-1}(c_\sigma) + \Sigma_{\tau \in u^{(2k-3)}} \Psi (\tau) \times \d \varphi^{-1}(b_\tau),
\]
where $c_\sigma$ is the center of $\sigma$ and $b_\tau$ is a wedge sum of edges connecting $c_{\tau}$ to each center of $(2k-2)$-simplex having $\tau$ as a face. Then we can simplicially attach 
\[
\Sigma_{\sigma \in u^{2k-2}} \Psi (\sigma) \times \varphi^{-1}(c_\sigma) + \Sigma_{\tau \in u^{2k-3}} \Psi (\tau) \times \varphi^{-1}(b_\tau)
\]
endowed with projections
\[
\Psi (\sigma) \times \varphi^{-1}(c_\sigma) \to \Psi (\sigma)
\]
and
\[
\Psi (\tau) \times \varphi^{-1}(b_\tau) \to \Psi (\tau)
\]
for each $\sigma$ and each $\tau$ to $M_{2k-2} \to u^{(2k-2)}$ along with
\[
\Sigma_{\sigma \in u^{(2k-2)}} \sigma \times \varphi^{-1}(c_\sigma).
\]
Then we can obtain a cobordism $W_{2k-2}$ over $BG$ between $M_{2k-2} \to BG^{(2k-2)}$ and $M_{2k-3} \to u^{(2k-3)} \subset BG^{(2k-3)}$. In this process, the subdivision we used depends on only dimension, the cylinder structure $P$ depends on only dimension, and the complexity of manifolds $\varphi^{-1}(c_\sigma)$ and $\varphi^{-1}(b_\tau)$ linearly depends on $M_{2k-2}$ whose complexity linearly depends on the complexity of $M$. Backward-inductively we can obtain a chain of cobordisms $W_i$ between over $BG$ between $M_{i} \to u^{(i)} \subset BG^{(i)}$ and $M_{i-1} \to u^{(i-1)} \subset BG^{(i-1)}$, where $i = 2k-3. 2k-4, \cdots, 2, 1, 0$. Also the complexity of each $W_i$ linearly depends on $M$. Thus the step~(\ref{step:5}) contributes only a linear bound for the complexity if $W$. (See~\cite{CL} for a fuller development of this idea.)

Since the complexity of gluing cobordisms is given by addition of their complexity, step~(\ref{step:6}) gives only a sum of complexity of every step. Using the cobordisms $W_{i}$ for $i=2k-2, 2k-3, \cdots, 2, 1$, we can glue them together to construct a cobordism $W$ over $G$ between $\sqcup M \to BG^{(2k-1)}$ and $N_0 \to BG^{(0)}$. As $BG$ is connected, $N_0 \to BG^{(0)}$ is homotopic to a constant map. We can modify the map $W \to BG$ on a collar neighborhood of $N_0$ using this homotopy to obtain a desired bordism over $G$ between $\sqcup M$ and a trivial end. Since adding manifolds together only adds complexity to the resulting cobordism, we have $\Delta (W) \leq C_{2k-1} \cdot N_{n,G} \cdot \Delta (M)$, where $C_{2k-1}$ is a constant depending only on the dimension $2k-1$.

Using the fact that $\Delta (\Tilde{W}) \leq \lvert G \rvert \cdot C_{2k-1} \cdot N_{n,G} \cdot \Delta (M)$, we can now apply Definition~\ref{def:ASrho} to obtain the following:
\begin{align*}
\lvert \rho_{g}(M) \rvert &=\lvert \frac{1}{N_{n,G}} \sign_{G}(g, \Tilde{W}) \rvert\\
 &\leq \frac{1}{N_{n,G}} \rank{H_{2k}(\Tilde{W})}\\
 &\leq \frac{1}{N_{n,G}} \frac{\binom{2k+1}{k+1}}{k+1} \cdot \lvert G \rvert \cdot C_{2k-1} \cdot N_{n,G} \cdot \Delta (M)\\
 &= \frac{\binom{2k}{k}}{k+1} \cdot \lvert G \rvert \cdot C_{2k-1} \cdot \Delta (M)
\end{align*}
\end{proof}

\section{Quantitative relative hyperbolization}\label{sec:hyperbolization}

In this section, we quantitatively study relative hyperbolization~\cite{DJW} and give the proof of Theorem~\ref{thm:CG}. We remark that these methods also lead to a proof of Theorem~\ref{thm:AS} with a constant that depends on the group $G$ but does not require the combinatorial work we did on $G$-colored polyhedra.

Following and geometrizing an idea of~\cite{Cha16} and~\cite{Lim}, the main theorem follows directly from the Cheeger-Gromov Atiyah-Patodi-Singer theorem (or alternatively from the topological definition of $\rho$-invariant given in~\cite{CW}) and the following:

\begin{lemma}\label{lem:geometric}
Given a PL $n$-manifold $M$, there is a cobordism $W$ from $M$ to $M$, and a quotient group $\Gamma$ of $\pi_1(W)$, so that (1) One component of $M$ has $\pi_1$ mapping trivially to $\Gamma$, (2) the other component is mapped injectively, and (3) the number of simplices in $W$ is linearly bounded by the number of simplices in $M$.
\end{lemma}

Recall that $\rho^{(2)}_{\Gamma}(M)$ is the difference between the $L^2$-signature of the $\Gamma$-cover of $W$ and its ordinary signature, as defined in Definition~\ref{def:CWrho}. Since the number of simplices in $W$ clearly bounds the signatures, it also bounds $\rho^{(2)}_{\Gamma}(M)$.

Our approach to construct the cobordism $W$ is based on the relative hyperbolization technique of Davis, Januszkiewicz, and Weinberger~\cite{DJW}. The construction of $W$ involves replacing each simplex in the barycentric subdivision of a cone $CM$ with a $\pi_1$-injective hyperbolized \enquote{simplex}, followed by the removal of a singularity at the cone point.

We begin by recalling key definitions and the Williams construction~\cite{Williams}, which is essential in our construction of hyperbolized $n$-simplices~\cite{DJW}. Let $\Delta^n$ denote the standard $n$-simplex.

\begin{definition}
Let $X$ be a topological space and $f:X \to \Delta^n$ be a continuous map. We say that the pair $(X,f)$ is a {\em space over the $n$-simplex}.
\end{definition}

\begin{definition}
If $L$ is an abstract simplicial complex and $g:L \to \Delta^n$ is a nondegenerate simplicial map, then the pair $(L, g)$ is called a {\em simplicial complex over $\Delta^n$}. Let both pairs $(L_1, g_1)$ and $(L_2, g_2)$ be simplicial complexes over $\Delta^n$. We call a nondegenerate simplicial map $g:L_1 \to L_2$ a {\em map over $\Delta^n$}, provided that the following diagram commutes:

\adjustbox{scale=1.3,center}{%
\begin{tikzcd}
L_1 \arrow[rd, "g_1"'] \arrow[rr, "g"] &   & L_2 \arrow[ld, "g_2"] \\
                                   & \Delta^n  &
\end{tikzcd}
}
We denote the category of simplicial complexes over $\Delta^n$ and maps over $\Delta^n$ by $\mathcal{K}(\Delta^n)$.
\end{definition}

\begin{remark}\label{rmk:derived}
The barycentric subdivision can be understood as a functor from the category of simplicial complexes and nondegenerate simplicial maps to $\mathcal{K}(\Delta^n)$. For an arbitrary abstract simplicial complex over $\Delta^n$ of dimension $n$, one can define a degree map $d:L \to \{0, 1, 2, \cdots, n\}$ which assigns to each simplex its dimension. Since the poset of nonempty subsets of $\{0, 1, 2, \cdots, n\}$ describes $\Delta^n$, the pair $(L', d)$ naturally becomes a simplicial complex over $\Delta^n$ where $L'$ is the barycentric subdivision of $L$. Furthermore, for a nondegenerate simplicial map $g:K \to L$ of $n$-dimensional complexes,
$g':K' \to L'$ is a map over $\Delta^n$. 
\end{remark}

We are now prepared to introduce the Williams construction. In essence, this procedure entails substituting each simplex in a given simplicial complex $L$ with a distinctive \enquote{simplex} $X$.

\begin{definition}\label{def:Williams}
\textbf{\textup{(Williams construction)}}
Let $(X, f)$ be a space over $\Delta^n$ and $(L, g)$ be a simplicial complex over $\Delta^n$. We define the \textit{fibered product} $X \widetilde{\triangle} L$ as the subspace of $X \times \mid L \mid$ consisting of all pairs $(x, y)$ such that $f(x)=g(y)$. Here, $\mid L \mid$ denotes the geometric realization of $L$. We use $f_L$ and $p$ to denote the natural projections:
\[
\adjustbox{scale=1.3,center}{%
\begin{tikzcd}
X \widetilde{\triangle} L \arrow[r, "f_L"] \arrow[d, "p"]
& \mid L \mid \arrow[d, "g"] \\
X \arrow[r, "f"]
& \Delta^n
\end{tikzcd}
}
\]
Suppose $L$ is an $n$-dimensional simplicial complex. Then, we define $X \triangle L$ as $X \widetilde{\triangle} L'$, where $L'$ denotes the barycentric subdivision of $L$, which is a simplicial complex over $\Delta^n$ as discussed in Remark~\ref{rmk:derived}.
\end{definition}

\begin{remark}\label{rmk:functor}
It is known that $X \widetilde{\triangle} - : \mathcal{K}(\Delta^n) \to Top$ is a functor from the category of simplicial complexes over $\Delta^n$ to the category of topological spaces~\cite{DJ}. By Remark~\ref{rmk:derived}, $X \triangle - : \mathcal{K}(\Delta^n) \to Top$ is also functorial.
\end{remark}

We define a hyperbolized \enquote{simplex} $(X,f)$, which will be used to replace each simplex in a simplicial complex $L$ for the purpose of hyperbolizing $L$.

\begin{definition}\label{def:hyperbolized simplex}
\textbf{\textup{(Hyperbolized $n$-simplex)}}
Let $(X,f)$ be a space over $\Delta^n$, where $X$ is a compact $n$-dimensional PL manifold with boundary and $f:X \to \Delta^n$ is piecewise linear such that the boundary operation is preserved under $f^{-1}$. That is, $\d (X_{\alpha})=X_{\d \alpha}$, where $\alpha$ is any $k$-dimensional face of $\Delta^n$ and $X_{\alpha}$ is $f^{-1}(\alpha)$, which is a $k$-dimensional PL submanifold of $\d X$. We say that $(X, f)$ is a {\em hyperbolized $n$-simplex}, provided that $X$ is a geodesic space of curvature less than or equal to $0$ and the subspace $X_J$ is totally geodesic for any connected subcomplex $J$ of $\Delta^n$, 
\end{definition}

\begin{remark}
For the application to a weak version of Theorem~\ref{thm:AS} (i.e. without any estimate on how constants depend on $G$), one can use cobordisms of multiples of $G$-labelled simplices to (the same multiple of) simplices that are trivially labelled. (Such cobordisms can be inferred from cobordism theory and do not need to be explicitly constructed.)
\end{remark}

\begin{remark}\label{rmk:hyperbolized simplex}
We construct hyperbolized $n$-simplices $(X^n, f)$ in every dimension $n$ using the Gromov's construction, following an inductive procedure. We start by taking $X^1$ to be the interval $\Delta^1$ together with the canonical map $f:X^1 \to \Delta^1$. Define a hyperbolized $n$-sphere $Y^n$ to be $X^n \triangle (\partial \Delta^{n+1})$, using the Williams construction. Consider an isometric reflection $r$ on $\Delta^{n+1}$. Then there is a half space $A \subset \Delta^{n+1}$ such that $\Delta^{n+1} = A \sqcup B \sqcup r(A)$ where $B=A \cap r(A)$. Now we define $X^{n+1} := Y^n \times [0,1] / \sim$ with the equivalence relation which is given by $r(a) \times \{1\} \sim r(a) \times \{0\}$ where $a \in A$. (Then, the boundary of $X^{n+1}$ is $\bigcup_{i=0,1} A \times \{i\} / \sim$ which is homeomorphic to $Y^n$.) Since $\Delta^{n+1}$ is the cone on $\d \Delta^{n}$, we can extend the nature map $\d X^{n+1} = Y^n \to \d \Delta^{n+1}$ to $f : X^{n+1} \to \Delta^{n+1}$, choosing a collared neighborhood of $\d X^{n+1}$ in $X^{n+1}$. Then, by construction, $(X^{n+1}, f)$ is a hyperbolized $n$-simplex. Notice that each face of the hyperbolized $n$-simplex is the hyperbolized $(n-1)$-simplices. For details we refer readers to~\cite[Section 4]{DJ}. This construction explicitly gives the number of $n$-simplices in a triangulation of $X^n$, which is $z(n):=3^{n-1} \cdot  n! \cdot  (n-1)!^2 \cdot  (n-2)!^2 \cdots (3!)^2 \cdot 2!$.
\end{remark}

By replacing each simplex in a simplicial complex with a hyperbolized simplex, we obtain a hyperbolization. In this paper, we use the hyperbolized simplices constructed in Remark~\ref{rmk:hyperbolized simplex}.

\begin{definition}
\textbf{\textup{(Hyperbolization)}} Given an $n$-dimensional simplicial complex $K$, and a hyperbolized $n$-simplex $(X^n, f)$ constructed as in Remark~\ref{rmk:hyperbolized simplex}, the {\em hyperbolization of $K$} is defined as $H(K) = X^n \triangle K$.
\end{definition}

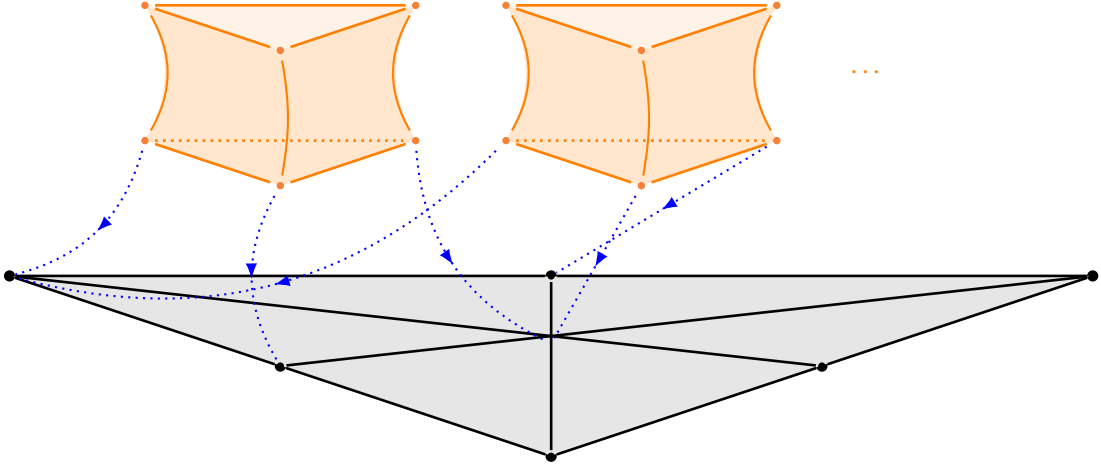
\begin{figure*}
\begin{center}
\tikzset{->-/.style n args={2}{decoration={
  markings,
  mark=at position #1 with {\arrow[line width=1pt]{#2}}},postaction={decorate}}}
\begin{tikzpicture}[scale=0.6, bullet/.style={circle,inner sep=1.5pt,fill}]
 
 \begin{scope}
  \path
   (-3,4) node[label=below:](B){}
   (0,3) node[label=below:](C){}
   (-3,1) node[label=below:](E){}
   (0,0) node[label=below:](F){}
   (3,4) node[label=below:](G){}
   (3,1) node[label=below:](I){}
   
   (5,4) node[label=below:](BB){}
   (8,3) node[label=below:](CC){}
   (5,1) node[label=below:](EE){}
   (8,0) node[label=below:](FF){}
   (11,4) node[label=below:](GG){}
   (11,1) node[label=below:](II){}
 
   (13,2.5) node[label=below:, orange](JJ){$\cdots$}
   
   (-6,-2) node[bullet,label=above:](A2){}
   (0,-4) node[bullet,label=below:](B2){}
   (6,-6) node[bullet,label=below:](C2){}
   (12,-4) node[bullet,label=below:](G2){}
   (18,-2) node[bullet,label=below:](H2){}
   (6,-2) node[bullet,label=below:](K2){}
   
   (6,-3.5) node[bullet,label=below:](J2){}
   ;
   {\fill[orange!10!] (-3,4) -- (0,3) -- (3,4) -- cycle;}
   {\fill[orange!20!] (-3,4) to [bend left=30] (-3,1) -- (0,0) -- (3,1) to [bend left=30] (3,4) -- (0,3) -- (-3,4);}
   
   {\draw [Orange,fill] (-3,4) circle [radius=0.07];}
   {\draw [Orange,fill] (0,3) circle [radius=0.07];}
   {\draw [Orange,fill] (-3,1) circle [radius=0.07];}
   {\draw [Orange,fill] (0,0) circle [radius=0.07];}
   {\draw [Orange,fill] (3,4) circle [radius=0.07];}
   {\draw [Orange,fill] (3,1) circle [radius=0.07];}
   
   {\draw [-,thick, orange] (B) to [bend left=30] (E);;}
   {\draw [-,thick, orange] (G) to [bend right=30] (I);;}
   {\draw [-,thick, orange] (C) to [bend left=10] (F);;}
   
   {\draw[line width=1pt, orange] (B) -- 
   node [text width=2.5cm,midway,below=0.1em,align=center ] {} (C);}
   {\draw[line width=1pt, orange] (E) --
   node [text width=2.5cm,midway,below=0.1em,align=center ] {} (F);}
   {\draw[line width=1pt, orange] (C) -- 
   node [text width=2.5cm,midway,below=0.1em,align=center ] {} (G);}
   {\draw[line width=1pt, orange] (F) --
   node [text width=2.5cm,midway,below=0.1em,align=center ] {} (I);}
   {\draw[line width=1pt, orange] (B) --
   node [text width=0.3cm,midway,below=0.5em,above=0.1em,align=center ] {} (G);}
   {\draw[dotted,line width=1pt, orange] (E) --
   node [text width=2.5cm,midway,above=0.1em,align=center ] {} (I);}

   {\fill[orange!10!] (5,4) -- (8,3) -- (11,4) -- cycle;}
   {\fill[orange!20!] (5,4) to [bend left=30] (5,1) -- (8,0) -- (11,1) to [bend left=30] (11,4) -- (8,3) -- (5,4);}

   {\draw [Orange,fill] (5,4) circle [radius=0.07];}
   {\draw [Orange,fill] (8,3) circle [radius=0.07];}
   {\draw [Orange,fill] (5,1) circle [radius=0.07];}
   {\draw [Orange,fill] (8,0) circle [radius=0.07];}
   {\draw [Orange,fill] (11,4) circle [radius=0.07];}
   {\draw [Orange,fill] (11,1) circle [radius=0.07];}
   
   {\draw [-,thick, orange] (BB) to [bend left=30] (EE);;}
   {\draw [-,thick, orange] (GG) to [bend right=30] (II);;}
   {\draw [-,thick, orange] (CC) to [bend left=10] (FF);;}
   
   {\draw[line width=1pt, orange] (BB) -- 
   node [text width=2.5cm,midway,below=0.1em,align=center ] {} (CC);}
   {\draw[line width=1pt, orange] (EE) --
   node [text width=2.5cm,midway,below=0.1em,align=center ] {} (FF);}
   {\draw[line width=1pt, orange] (CC) -- 
   node [text width=2.5cm,midway,below=0.1em,align=center ] {} (GG);}
   {\draw[line width=1pt, orange] (FF) --
   node [text width=2.5cm,midway,below=0.1em,align=center ] {} (II);}
   {\draw[line width=1pt, orange] (BB) --
   node [text width=0.3cm,midway,below=0.5em,above=0.1em,align=center ] {} (GG);}
   {\draw[dotted,line width=1pt, orange] (EE) --
   node [text width=2.5cm,midway,above=0.1em,align=center ] {} (II);}

   {\draw[fill=black!10!] (-6,-2) -- (18,-2) -- (6,-6) -- cycle;}
   
   {\draw [black,fill] (-6,-2) circle [radius=0.07];}
   {\draw [black,fill] (0,-4) circle [radius=0.07];}
   {\draw [black,fill] (6,-6) circle [radius=0.07];}

   {\draw [black,fill] (12,-4) circle [radius=0.07];}
   {\draw [black,fill] (18,-2) circle [radius=0.07];}
 
   {\draw [black,fill] (6,-2) circle [radius=0.07];}
   
   {\draw[line width=1pt] (A2) -- 
   node [text width=2.5cm,midway,above=0.1em,align=center ] {} (B2);}
   {\draw[line width=1pt] (B2) -- 
   node [text width=2.5cm,midway,above=0.1em,align=center ] {} (C2);}

   {\draw[line width=1pt] (C2) -- 
   node [text width=2.5cm,midway,above=0.1em,align=center ] {} (G2);}
   {\draw[line width=1pt] (G2) -- 
   node [text width=2.5cm,midway,above=0.1em,align=center ] {} (H2);}

   {\draw[line width=1pt] (A2) --
   node [text width=2.5cm,midway,above=0.1em,align=center ] {} (K2);}
   {\draw[line width=1pt] (K2) --
   node [text width=,midway,above=0.1em,align=center ] {} (H2);}
   
   {\draw[line width=1pt] (B2) --
   node [text width=2.5cm,midway,above=0.05em,align=center ] {} (H2);}  
   {\draw[line width=1pt] (K2) --
   node [text width=2.5cm,midway,above=0.1em,align=center ] {} (C2);}
   {\draw[line width=1pt] (A2) --
   node [text width=0.3cm,midway,below=0.5em,above=0.1em,align=center ] {} (G2);}
   
   {\draw [->-={0.5}{latex}, thick, dotted, blue] (E) to [bend left=30] (A2);;}
   {\draw [->-={0.5}{latex}, thick, dotted, blue] (F) to [bend right=30] (B2);;}
   {\draw [->-={0.5}{latex}, thick, dotted, blue] (I) to [bend right=30] (J2);;}

   {\draw [->-={0.5}{latex}, thick, dotted, blue] (EE) to [bend left=30] (A2);;}
   {\draw [->-={0.5}{latex}, thick, dotted, blue] (FF) to [bend left=0] (J2);;}
   {\draw [->-={0.5}{latex}, thick, dotted, blue] (II) to [bend left=0] (K2);;}

\end{scope} 
\end{tikzpicture} 
\end{center}\caption{The hyperbolization process is to fiber-wisely replace an abstract simplex in a barycentric subdivision of a simplicial complex by a hyperbolized simplex. This picture describes $H(\sigma)=X^2 \triangle \sigma =X^2 \widetilde{\triangle} \sigma'$ where $\sigma$ is an abstract $2$-simplex. The hyperbolic $2$-simplex is colored orange, while the barycentric subdivision of $\sigma$ is colored gray.}\label{figure:hyperbolization}
\end{figure*}

We define the concept of relative hyperbolization. 

\begin{definition}
\textbf{\textup{(Relative hyperbolization)}}
Let $K$ be a simplicial complex and $L=\sqcup L_i$ be a subcomplex of $K$ where $\{L_i\}$ be the set of path components of $L$. By attaching a cone $CL_i$ to $K$ on each $L_i$, we obtain the simplicial complex $K \cup CL$. We denote the cone point corresponding to $L_i$ by $l_i$. By hyperbolizing $K \cup CL$, we obtain $H(K \cup CL)$. The link of $l_i$ in $H(K \cup CL)$, denoted by $lk_i$, can be identified with a subdivision of $L_i$.

We define the {\em relative hyperbolization of $K$ with respect to $L$} as the space $J(K,L)$ obtained by removing a small open conical neighborhood of each $l_i$ from $H(K \cup CL)$. Since the boundary of such a neighborhood is $lk_i (= L_i)$, we have found a subspace of $J(K,L)$ which comprises a copy of $L$.
\end{definition}

We now show Lemma~\ref{lem:geometric}.

\begin{proof}[Proof of Lemma~\ref{lem:geometric}]
Let $M$ be a PL manifold $M$. Consider $M \times [0,1]$. Let $L_0$ and $L_1$ be the boundary components $M \times \{0\}$ and $M \times \{1\}$ respectively. Denote $L = L_0 \sqcup L_1$. Now consider a hyperbolization $H(M \times [0,1] \cup CL)$. Let $W$ be the relative hyperbolization $J(M \times [0, 1], L)$. Since we obtain $W$ by removing small open conical neighborhoods of $l_0$ and $l_1$ which are cone points of $M \times \{0\}$ and $M \times \{1\}$ respectively, the corresponding boundaries $lk_0$ and $lk_1$ are homeomorphic to $M \times \{0\}$ and $M \times \{1\}$ respectively. Now, let $\Gamma$ be $\pi_1(W \cup Clk_1)$.

Since the universal cover $\widetilde{H}(M \times [0,1], CL)$ of $H(M \times [0,1], CL)$ is a $\CAT(0)$ space (see~\cite[4c.2]{DJ}), there is a deformation retraction of $\widetilde{H}(M \times [0,1], CL) - \widetilde{l_i}$ onto $\widetilde{L_i}$ for each $i=0,1$ which is given by a geodesic contraction (see~\cite[Lemma 2.2]{DJW}). Notice that the restriction of the deformation retraction to $\widetilde{W}=\widetilde{J}(M \times [0,1],L)$ is a $\pi_1$-injective map. Hence $\pi_1(lk_i) \to \pi_1(W)$ is injective for each $i=0,1$.

Now consider $W \cup Clk_1$ obtained by attaching the cone $Clk_1$ of $lk_1$ to the bottom boundary $lk_1$ of $W$. Then $\Gamma = \pi_1(W \cup Clk_1)$ is a quotient group of $\pi_1(W)$. i.e. $\pi_1(lk_1) \to \Gamma$ is the zero map. We show $\pi_1(lk_0) \to \Gamma$ is injective. For a contradiction, assume that there are nontrivial elements $[\alpha] \in \pi_1(lk_0)$ and $[\beta] \in \pi_1(lk_1)$ such that $[\alpha]=[\beta]$ in $\pi_1(W)$. By the construction of $W$, this means that there are nontrivial elements $[\alpha'] \in \pi_1(H(M \times \{0\}))$ and $[\beta'] \in \pi_1(H(M \times \{1\}))$ such that $[\alpha']=[\beta']$ in $\pi_1(H(M \times [0,1]))$. In other words, there is a cylinder $N$ in $H(M \times [0,1])$ between $\alpha'$ and $\beta'$. Since the canonical projection $p:H(M \times [0,1]) \to M \times [0,1]$ is continuous, $p(N)$ is a cylinder in $M \times [0,1]$ between the two loops $p(\alpha')$ in $M \times \{0\}$ and $p(\beta')$ in $M \times \{1\}$. Recall that the hyperbolization replaces each $k$-simplex in the given triangulation of $M \times [0,1]$ with the hyperbolized $k$-simplex in Remark~\ref{rmk:hyperbolized simplex} for $k=1,2,\cdots,n+1$ (see~\cite[Section 4]{DJ}). Specially $1$-simplices in $M$ are remained in $H(M)$ in the hyperbolization. By the canonical projection in the Gromov's construction, we obtain $N \subset H(p(N)) \subset H(M \times [0,1])$. Since $H(p(N))$ is homeomorphic to a two-punctured $g$-torus that $\alpha'$ and $\beta'$ represent the two non-homotopic loops of the punctured parts, it is a contradiction to that $[\alpha']=[\beta']$ in $\pi_1(H(M \times [0,1]))$. Then, $\pi_1(lk_0) \to \Gamma$ is injective.

We then obtain a cobordism $W$ between two boundary components $lk_0$ and $lk_1$ which are homeomorphic to $M$, and a commutative diagram below:
\[
\adjustbox{scale=1.3, center}{%
\begin{tikzcd}
\pi_1(lk_0) \arrow[rd, hook, "i_{0\ast}"] \arrow[rrd, hook, "i_{W\ast} \circ i_{0\ast}"] & & \\
& \pi_1(W) \arrow[r, "i_{W\ast}"] & \Gamma=\pi_1(W \cup Clk_1) \\
\pi_1(lk_1) \arrow[ru, hook, "i_{1\ast}"] \arrow[rru, "\times 0" ']
&  &
\end{tikzcd}
}
\]
where $i_{0\ast}$, $i_{1\ast}$, and $i_{W\ast}$ are induced maps from $M \times \{0, 1\} \hookrightarrow W \hookrightarrow H(M \times [0, 1]) \cup C(M \times \{1\})$.
\end{proof}

We give the proof of Theorem~\ref{thm:CG}.

\begin{proof}[Proof of Theorem~\ref{thm:CG}]
By Lemma~\ref{lem:geometric}, $\d W=M \times \{0\} \sqcup M \times \{1\}$ over $\Gamma$. Then, the $L^2$ signature defect of $W$ is $\rho^{(2)}_{\Gamma}(M \times \{0\})-\rho^{(2)}_{\Gamma}(M \times \{1\})$. Since $M \times \{1\}$ is $\pi_1$-trivially over $\Gamma$, $\rho^{(2)}_{\Gamma}(M \times \{1\})=0$. Since the signatures of $W$ are bounded by the rank of its middle dimension homology, 
$\rho^{(2)}_{\Gamma}(M \times \{0\})$ is bounded by twice the number of the middle dimension simplices of $W$. One can obtain $W$ as a subcomplex of $H(SM)$, where $SM$ is the suspension of $M$. Then, the number of $4k$-simplices of $H(SM)$ is $2 \cdot z(4k) \cdot (4k+1)! \cdot \Delta(M)$, where $\Delta(M)$ is the number of $(4k-1)$-simplex of $M$ and $z(x)$ is the function introduced in Remark~\ref{rmk:hyperbolized simplex}. By bounding the number of $2k$-simplices of $W$, we obtain 
\[
\displaystyle{\lvert \rho^{(2)}_{\Gamma}(M) \rvert \leq \frac{2}{2k+1} \cdot {{4k+1}\choose{2k+1}} \cdot 2 \cdot z(4k) \cdot (4k+1)! \cdot \Delta(M)}.\qedhere
\] 
\end{proof}

We end this section by noting that one can extend the definition of the Cheeger-Gromov invariant regarding an arbitrary representation of $\pi_1(M)$ instead of the injection of $\pi_1(M)$ into $\Gamma$. The first author of this paper and Cha will independently show that there exists \emph{universal} linear bounds on the generalized Cheeger-Gromov invariants in a future paper~\cite{CL}.

\section{Complexity and homotopy structures of homotopy lens spaces}\label{sec:complexity}

In this section we prove Theorem~\ref{cor:CG}, Theorem~\ref{cor:AS}, and Theorem~\ref{cor:bddgeo}.

As an application of our results, we can use the new bounds for $\rho$-invariants of high-dimensional manifolds to study the complexity theory of manifolds, particularly for lens spaces and homotopy lens spaces. The complexity of a 3-manifold is defined as the smallest number of 3-simplices required to represent the manifold as a quotient space of simplices. While this definition is natural and clear, computing the complexity is not easy. For instance, it is well-known that there are only finitely many 3-manifolds whose complexity is less than or equal to a given natural number. However, classifying 3-manifolds based on complexity remains a challenging problem. Matveev~\cite{Mat} extensively studied the complexity of 3-manifolds. Lackenby and Purcell~\cite{LP} demonstrated its relationship to other key topological and geometric quantities from the geometry of the mapping class group and Teichm\"{u}ller space. Costantino~\cite{Cos06} introduced the 4-dimensional analogue of Matveev's complexity of 3-manifolds, using Turaev shadows. However, studying the complexity of high-dimensional manifolds remains a mysterious topic, and there are only a few example. Our main results on high-dimensional manifolds provide a new application in the complexity theory of high-dimensional lens spaces and homotopy lens spaces.

Our first application confirms the high-dimensional analogue of the conjecture by Matveev~\cite{Mat} and Jaco, Rubinstein, and Tillmann~\cite{JRT09} for specific high-dimensional lens spaces, up to a bounded multiplicative error which depends only on the dimension. Matveev~\cite{Mat} and Jaco, Rubinstein, and Tillmann~\cite{JRT09} independently conjectured that the complexity of $L^3(N;1)$ is $N-3$. In~\cite{JRT09}, they proved the upper bound and the case for even $N$. However, the lower bound remained a challenging open problem. Cha~\cite{Cha16} provided a lower bound for $L^3(N;1)$ by employing the linear upper bounds for $\rho$-invariants. In~\cite{LP}, Lackenby and Purcell obtained this result for general $L^3(p;q)$ in terms of continued fractions, using the geometry of the mapping class group. Theorem~\ref{cor:CG} extends the conjecture for high-dimensional lens spaces and confirms its analogue for general $L^{2d-1}(N;1,1, \cdots, 1)$, providing new examples for the complexity theory of high-dimensional lens spaces.

Now we give a proof of Theorem~\ref{cor:CG}: the number of simplices of the standard lens space $L_{N}(1,1,\cdots,1)$ of dimension $2d-1$ (with fundamental group $\Z_N$) grows (in $N$) like $N^{d-1}$, i.e. the number of simplices is bounded above and below by dimensional constants times $N^{d-1}$.

\begin{proof}[Proof of Theorem~\ref{cor:CG}]
For the upper bound, we use an inductive construction. Any $\Z_N$-quotient of $S^1$ is $S^1$. One can obtain $L^{2(d+1)-1}(N;1,1, \cdots, 1)$ as a $\Z_N$-quotient of the join of the universal cover of $L^{2d-1}(N;1,1, \cdots, 1)$ and a $N$-gon. Thus, we can inductively obtain the upper bound.

For the lower bound, we first consider the case that $d$ is even. Atiyah and Bott~\cite[Theorem 6.27]{AB} computed their $\rho$-invariant for general lens spaces in terms of $G$-signatures. For the regular representation $\alpha$ of $\pi_1(L^{2d-1}(N;1,1, \cdots, 1))=\Z_N$, the theorem gives the following:
\[
\rho_{\alpha}(L^{2d-1}(N;1,1, \cdots, 1))=\Sigma_{k=1}^{N-1}{\textup{cot}}^{d}(\frac{\pi k}{N}).
\]
For $N \geq 4$, $(\frac{1}{\pi}N)^d < \rho_{\alpha}(L^{2d-1}(N;1,1, \cdots, 1))$. Combining it with Theorem~\ref{thm:AS}, we obtain
\[
C \cdot N^{d-1} < \Delta (M)
\]
where $C$ is a constant that depends on the dimension $2d-1$.

For a given odd $d$, we obtain the lower bound by the above construction of $L^{2(d+1)-1}(N;1,1, \cdots, 1)$ as a $\Z_N$-quotient of the join of the universal cover of $L^{2d-1}(N;1,1, \cdots, 1)$ and a $N$-gon.
\end{proof}

Another application of our bounds for the Wall $\rho$-invariants is to show that the number of simplices needed for an element of the homotopy structure set of a lens space~\cite[Section 14E]{Wall99}, is approximately the same as its distance from the origin in that set, thought of as an abelian group. This is based on the connection between $\rho$-invariants and surgery, with a trick from~\cite{Wei82} being used to construct that many homotopy lens spaces with $V$ simplices, and the inequality in this paper to give an upper bound on the $\rho$-invariants of such manifolds.

A consequence of surgery theory (the Rothenberg sequence, see~\cite[Section 17]{Moore}) and the h-cobordism theorem is that the h-cobordism classes of manifolds homotopy equivalent to $M$ and the manifolds simple homotopy equivalent to $M$ are essentially the same, up to \emph{mod 2 torsion phenomena}\footnote{There is a map $S^s(L) \to S^h(L)$ with kernel and cokernel elementary abelian $2$-groups.} which are finite whenever the Whitehead group of the fundamental group is finitely generated, and in particular, for finite fundamental group.

The manifolds simple homotopy equivalent to a lens spaces, are determined by the $\rho_g$ invariants
for $g \neq e$ (note that $\rho_g$ and $\rho_{g^{-1}}$ are essentially equivalent). For $g$ an element of order $2$, there is
such an invariant in dimensions $3$ mod $4$, but not $1$ mod $4$. (In $1$ mod $4$, the bounding manifold will have a skew symmetric invariant inner product, but there are no interesting invariants of
these for automorphisms of order $2$.)

As a result, one has for $N$ odd, $[(N-1)/2]$ invariants of these manifolds. For $N$ even, there are $[(N-1)/2]$ such invariants in all odd dimensions and 1 extra in dimension 3 mod 4. This can be
summarized as $[(N-1)/2] + \delta(N,d)$ where $\delta(N,d) = 0$ unless $N$ is even and $d$ is $3$ mod $4$, in which case it is $1$. 

Now we give a proof of Theorem~\ref{cor:AS}: the number of h-cobordism classes of manifolds simple homotopy equivalent to a lens space $L_N^d$ with at most $V$ simplices is bounded between two constant multiples (depending on $d$)$\times V^{(N-1)/2 + \delta (N,d)}$ where $\delta (N,d) = 0$ unless $N$ is even and $d$ is $3$ mod $4$, and in that case it is $1$.

\begin{proof}[Proof of Theorem~\ref{cor:AS}]
    Let $M$ be a lens spaces $L_N^d$. We use the Wall realization $L_{2d}(\Z/N) \to S(M)$~\cite[Theorem 5.8]{Wall99}, which identifies all elements in $L_{2d}(\Z/N)$ as obstructions to surgery problems in $S(M)$, to obtain a $\Z^{(N-1)/2}$ number of such  manifolds, detected by the $\rho_{G}$ invariant. This $L$-group is finitely generated. Pick a codimension  one submanifold $X$ of $M$ so that $\pi_1(X) \to \pi_1(M)$ is an isomorphism. Then one can lift the map  $L_{2d}(\Z/N) \to S(M)$ through $L_{2d}(\Z/N) \to S(X \times [0,1] \textup{ rel } \d) \to S(M)$. One cuts along $X$ and glues in between the two sides the Wall realization mapping to $X \times [0,1] \textup{ rel } \d$~\cite{Wei82}. The constant for the lower bound is the reciprocal of  the largest number of simplices in any of the Wall realizations for any of the generators with a  negative additive constant for the number of simplices in a triangulation of $M$ in which $N$ is a  subcomplex. In this construction, the codimension one submanifold $X$ we use is the boundary of regular neighborhood of a lower dimensional lens space. If $d=3$, then the submanifold $X$ has dimension $4$. In that case one might have to connect sum some $S^2 \times S^2$ to achieve the Wall realization. (See~\cite{CS71}.)

    For the upper bound, Theorem~\ref{thm:AS} plays a key role. As the above discussion, we have $[(N-1)/2] + \delta(N,d)$ invariants where $\delta(N,d) = 0$ unless $N$ is even and $d$ is $3$ mod $4$, in which case it is $1$. Since we showed the upper bound of $\rho$-invariant $\rho_{g}$ linearly depends on the complexity of the given homotopy lens space in Theorem~\ref{thm:AS}, the number of h-cobordism classes of manifolds homotopy equivalent to a lens space $L_{N}^{2d-1}$ with $V$ simplices is bounded above by $V^{(N-1)/2 + \delta (N,d)}$ (in $V$) with a multiplicative coefficient which depends only on dimension.
\end{proof}

The difference between homotopy equivalence and simple homotopy equivalence is governed by the Whitehead group. For a finite group $G$, the Whitehead group $Wh(G)$ is finitely generated and detected by taking the determinants of real representations. In other words, a representation gives a map $\Z G$ to $GL_n(\C)$, and one can take the absolute value of the determinant of a matrix representing an element of the Whitehead group. An element of $Wh(G)$ is finite order if these determinants have determinant $1$ for all the (finitely many) irreducible representations of $G$. (For details, we refer readers to~\cite{Bass}.) However, when a sum of representations is rational, then the determinant will be root of unity. In other words, this abelian group has rank equal to the number of irreducible real representation minus the number of rational representations. The rational representations are in a $1-1$ correspondence with the divisors of $N$, the number of which we call $d(N)$. For $N = 2k+1$, this rank is $k+1-d(N)$ and for $N = 2k$ this is $k+1-d(N)$.

Combining the torsion and the $\rho$-invariants, the number of invariants we have for homotopy lens spaces, for $N = 2k+1$, is
thus $k+k+1-d(N) = N-d(N)$. For $N = 2k$, we obtain $k-1 + \delta + k+1 - d(N) = N - d(N) + \delta$.

Now we show Theorem~\ref{cor:bddgeo}: the number of homotopy lens spaces with bounded geometry with fundamental group $\Z/N$ in dimension $d$ and volume $V$ can be bounded above and below by constants times $V^{N-d(N)+\delta(N,d)}$.

\begin{proof}[Proof of Theorem~\ref{cor:bddgeo}]
There are only finitely many homotopy types of $d$-dimensional lens spaces with fundamental group $\Z/N$, so we are not harmed by lumping them all together. The upper bounds follow from our bounds on $\rho$ and torsions. The lower bounds follow from two analogous constructions.    

For realizing the torsions, we note that in the above notation, the subgroup of the Whitehead group $2Wh(\Z/N)$ (which is a free abelian group of the same rank as $Wh(\Z/N)$) is realized by h-cobordisms of $N$ to itself (see~\cite{Milnor66}). As before \emph{cutting and pasting} along the boundary of a regular neighborhood of the $2$-skeleton give a construction of the right number of simple homotopy types.

To get the estimate on the torsion, we use the Reidemeister torsion which detects the Whitehead torsion, modulo elements of finite order~\cite{Milnor66}:

\begin{proposition}\label{prop:diff}
    The Reidemeister torsion of a simplicial complex of bounded geometry with $S$ simplices is $\exp(O(S))$.
\end{proposition}

This follows from the definition of the Reidemeister torsion as\cite{Ray-Singer71}:
\[
Exp(\Sigma(-l)^{q+1} \cdot q \cdot \log( \det(\Delta_q)))
\]
The logarithm of this gives an additive injection (as one varies over irreducible flat bundles) of  $Wh(\Z/N)$ into $\R^{[N/2]}$ (whose image is a lattice). Since the complex has bounded geometry the  determinants of the relevant Laplacians trivially grow like $exp(S)$ giving the result. 

The constructions do not interfere with each other, so we obtain the full number of manifolds desired.
\end{proof}

\section{Density of Cheeger-Gromov $L^2$ $\rho$-invariants over structure sets}\label{sec:density}

We end this paper by presenting the proof of Theorem~\ref{thm:CW}: suppose $M$ is a closed oriented manifold of of dimension $4k+3$, where $k>0$, and $\pi = \pi_{1}(M)$ has finite subgroups of arbitrary large order. Then the values of $\rho^{(2)}(M’)$ as $M'$ varies over manifolds homotopy equivalent to $M$ is a dense subset of $\R$.  This implies that the group $S(M)$ of such manifolds (studied by surgery theory) is not finitely generated.

This section can be viewed as essentially independent of the rest of the paper.

\begin{proof}[Proof of Theorem~\ref{thm:CW}]
The proof is a variant of the theorem of~\cite{CW} and uses the surgery exact sequence and some related maps regarding an arbitrary finite subgroup $G$ of $\pi$:
\[
\adjustbox{scale=1,center}{%
\begin{tikzcd}
& L_{0}(G) \arrow[d, ""] &\\
H_{0}(M;\textbf{L}) \arrow[r, ""] & L_{0}(\pi) \arrow[r, ""] \arrow[d, ""] & S(M) \arrow[d, ""]\\
& \R \arrow[r, "="]
& \R
\end{tikzcd}
}
\]
The map $L_{0}(\pi)$ to $\R$ assigns to a symmetric quadratic form over $\Z{\pi}$ the difference between its $L^2$ signature and its ordinary signature since the $L^2$ signature is natural with respect to inclusion and there is also an ordinary signature obtained by mapping a group to the trivial group, and then just taking the signature. The composite with the map $H_{0}(M;\textbf{L}) \to L_{0}(\pi)$ is trivial by Atiyah’s theorem, which generalizes the multiplicativity property of ordinary signature for finite coverings to infinite ones~\cite{A76}. The map $L_{0}(\pi) \to S(M)$ is the Wall realization, and the commutativity of the square is simply the fact that  $(\rho^{(2)}(M’)-\rho^{(2)}(M))-(\rho(M’)-\rho(M))$ is the same as difference of signatures arising in the normal cobordism between $M’$ and $M$.  Now if $G$ is finite group, Wall~\cite[Theorem 13A.4]{Wall99} showed that the transfer map $L_{0}(G) \to L_{0}(e)$ is surjective. It is $\lvert G \rvert$ times the $L^2$ signature when $G$ is finite. This factor gives rise to denominators in the rho invariant.

Since the composite with the inclusion $L_{0}(e) \to L_{0}(G) \to L_{0}(e)$ is multiplication by $\lvert G \rvert$, we can find an element which has signature $0$, and whose $G$-fold cover has signature $8$. (In $L$-theory, the image of trivial group $L$ in $\Z$ is the multiples of $8$ because the quadratic forms that arise in L-theory always have even entries on the diagonal, which forces the signature to be a multiple of $8$ as the generator of $L_0(e)$ is the $E_8$ quadratic form.) This element has gives realization of the element $8/{\lvert G \rvert}$ and all its multiples. As, by hypothesis, $\pi_{1}(M)$ has subgroups of arbitrarily large order, this image is dense. 

For each $G$ the image of $L_{0}(G) \to L_{0}(\pi) \to S(M)$ is a subgroup, and as $\lvert G \rvert$ increases this set of subgroups is evidently not stabilizing (as their images in $\R$ are not), so $S(M)$ cannot be a finitely generated group.
\end{proof}

\begin{remark}
We have shown that for such a group, the L-groups $L_0(ZG)$ and $L_0(RG)$ are infinitely generated (even modulo torsion).  The same argument applies to $K_0(C^{\ast}G)$ which can be used to prove an analogue of our theorem for metrics of positive scalar curvature on appropriate spin manifolds. In other words, the concordance classes of metrics of positive scalar curvature on a spin manifold has an abelian group structure (see~\cite[Section 4]{WY15}). We can get a lower bound on this group by a similar method, using an $L^2$ rho invariant associated to the Dirac operator.
\end{remark}

\bibliographystyle{abbrv}
\bibliography{research.bib}

\end{document}